\documentclass[12pt,reqno]{amsart}
\usepackage{hyperref}
\usepackage{soul}
\usepackage{amsmath,amssymb}
\usepackage[english]{babel}
\usepackage{caption}
\usepackage{amssymb,amsmath,euscript,enumerate,tikz}
\usepackage[margin=1in]{geometry}
\usepackage{graphicx}
\usepackage{float}
\usepackage{pgf,tikz}
\usepackage{mathrsfs}
\usepackage{upgreek}

\newcommand \Tor{\operatorname{Tor}}

\newcommand \iv{\operatorname{iv}}
\newcommand \pd{\operatorname{pd}}

\newcommand \h{\operatorname{ht}}

\newcommand \depth{\operatorname{depth}}
\newcommand \gr{\operatorname{gr}}

\newcommand \R{\mathcal{R}}
\newcommand \p{\mathfrak{p}}

\newcommand \K{\mathbb{K}}

\newcommand{\mS}{\mathcal{S}}

\newcommand{\KK}{\mathbb{K}}

\renewcommand{\vec}[1]{\mathbf{#1}}
\newtheorem{theorem}{Theorem}[section]
\newtheorem{definition}[theorem]{Definition}
\newtheorem{lemma}[theorem]{Lemma}
\newtheorem{proposition}[theorem]{Proposition}

\newtheorem{remark}[theorem]{Remark}
\newtheorem{corollary}[theorem]{Corollary}

\newtheorem{con}[theorem]{Conjecture}

\begin{document}
\title[Lov\'asz-Saks-Schrijver ideals and parity binomial edge ideals of graphs]{Lov\'asz-Saks-Schrijver ideals and parity binomial edge ideals of graphs}
\author[Arvind Kumar]{Arvind Kumar}
\email{arvkumar11@gmail.com}
\address{Department of Mathematics, Indian Institute of Technology
Madras, Chennai, INDIA - 600036}

\begin{abstract}
Let $G$ be a  simple graph on $n$ vertices. Let $L_G \text{ and } \mathcal{I}_G \:  $ denote the
 Lov\'asz-Saks-Schrijver(LSS) ideal and parity binomial edge ideal of $G$ in the polynomial ring $S = \K[x_1,
\ldots, x_n, y_1, \ldots, y_n] $ respectively.  We  classify  graphs whose LSS
ideals and parity binomial edge ideals are  complete intersections.  We  also classify  graphs whose LSS
ideals and parity binomial edge ideals are  almost complete intersections, and we prove that their Rees
algebra  is Cohen-Macaulay. We compute the second graded 
Betti number and obtain a minimal presentation of LSS ideals of  trees and odd
unicyclic graphs.  We also
obtain an explicit description of the defining ideal of the symmetric 
algebra of LSS ideals of  trees and odd unicyclic graphs. 
\end{abstract}
\keywords{Lov\'asz-Saks-Schrijver(LSS) ideal, parity binomial edge ideal, binomial edge ideal, syzygy, Rees Algebra, Betti number, complete intersection, almost complete intersection}
\thanks{AMS Subject Classification (2010): 13A30, 13D02,13C13, 05E40}
\dedicatory{Dedicated to Professor J\"urgen Herzog on the occasion of his 80th birthday}
\maketitle
\section{Introduction}

Let $\K$ be any field. Let $G$ be a  simple graph with $V(G) =[n]:= \{1, \ldots, n\}$. We study the following four classes of ideals associated with the graph $G$: 
\begin{itemize}
	\item \textbf{Binomial Edge Ideals:} Herzog et al. in \cite{HH1} and independently 
	Ohtani in \cite{oh} defined the {\it binomial edge ideal} of $G$ as $$J_G = (
	x_i y_j - x_j y_i ~ : i < j, \; \{i,j\} \in E(G)) \subset
	\K[x_1, \ldots, x_{n}, y_1, \ldots, y_{n}].$$ 
	\item \textbf{Lov\'asz-Saks-Schrijver ideals:} Let $d \geq 1$ be an integer. The ideal $$L_G^{\K}(d)= \left(\sum_{l=1}^d x_{il}x_{jl} : \{i,j\} \in E(G)\right)\subset \K[x_{kl}: 1\leq k\leq n, 1 \leq l \leq d]$$ is known as {\it Lov\'asz-Saks-Schrijver ideal} of the graph $G$ with respect to $\K$. The set of all orthogonal representation of the complementary graph of ${G}$ is the zero set  of the ideal $L_G^{\K}(d)$ in $\K^{n \times d}$. We refer the reader to \cite{lss89,lov79} for more on the orthogonal representation of graphs. In this article, we set $L_G :=L_G^{\K}(2)$.
	\item \textbf{Permanental Edge Ideals:} In \cite{her3}, Herzog et al. introduced the notation of permanental edge ideals of graphs.	
	The {\it permanental edge ideal} of a graph $G$ is denoted by $\Pi_G$ and it is defined as 
	$$\Pi_G =(x_iy_j+x_jy_i : \{i,j\} \in E(G)) \subset \K[x_1,\ldots,x_n,y_1,\ldots,y_n].$$
	\item \textbf{Parity Binomial Edge Ideals:} Kahle et al. in \cite{KCT} introduced the notion of parity binomial edge ideals of graphs. The {\it parity binomial edge ideal} of a graph $G$ is defined as $$
	\mathcal{I}_G=(x_ix_j-y_iy_j : \{i,j\} \in E(G)) \subset \K[x_1,\ldots,x_n,y_1,\ldots,y_n].$$ 
\end{itemize} 
In the recent past,
researchers have been trying to understand the connection between
combinatorial invariants of $G$ and algebraic invariants of $J_G$. The connection between the combinatorial properties of $G$ and the algebraic properties $J_G$ has been established by many authors, see \cite{her1,HH1,JACM,JAR, KM3, MM, KM12} for a partial
list. For $d=1$, the Lov\'asz-Saks-Schrijver ideal of a graph $G$ is a monomial ideal known as the edge ideal of graph $G$. The algebraic properties of edge ideals of graphs are well understood, see \cite[Chapter 9]{Mon}. For $d=2$, the Lov\'asz-Saks-Schrijver ideal of a graph $G$ is a binomial ideal defined as $L_G=(x_ix_j+y_iy_j: \{i,j\} \in E(G)) \subset \K[x_1,\ldots,x_n,y_1,\ldots,y_n]$. In \cite{her3}, Herzog et al. proved that if char$(\K) \neq 2$, then $L_G$ is a radical ideal. Also, they computed the primary decomposition of $L_G$ when $\sqrt{-1} \notin \K$ and char$(\K)\neq 2$. In \cite{AV19}, Conca and Welker studied the algebraic properties  of $L_G^{\K}(d)$. They proved that $L_G^{\K}(2)$ is complete intersection if and only if  $G$ does not contain claw  or even cycle (\cite[Theorem 1.4]{AV19}). Also, they proved that $L_G^{\K}(3)$ is prime if and only if $G$ does not contain claw or $C_4$. More precisely, in \cite{AV19}, Conca and Welker analyzed the question ``When is $L_G^{\K}(d)$ radical, complete intersection or prime''? In \cite{her3}, Herzog et al. computed  Gr$\ddot{\text{o}}$bner basis of permanental edge ideals of graphs. Also,  they proved that permanental edge ideal of a graph is a radical ideal, in \cite{her3}.  
In \cite{KCT}, Kahle et al. studied the algebraic properties such as primary decomposition, mesoprimary decomposition, Markov bases and radicality of parity binomial edge ideals. However, nothing is known about the algebraic properties such as complete intersection, almost complete intersection, Rees algebra, symmetric algebra and Betti numbers of parity binomial edge ideals.
In this article, we focus on the algebraic properties  such as almost complete intersection, projective dimension, Rees algebra, symmetric algebra and Betti numbers of $L_G, \Pi_G \text{ and } \mathcal{I}_G$. It was proved by Bolognini et al \cite[Corollary 6.2]{dav} that if $G$ is a bipartite graph, then $ L_G, \; \Pi_G \text{ and } \;\mathcal{I}_G$ are essentially same as $J_G$. In \cite{PZ}, Schenzel and Zafar studied the algebraic properties of complete bipartite graphs. In \cite{dav}, Bolognini et al. studied the Cohen-Macaulayness of binomial edge ideal of bipartite graphs. The algebraic properties of Cohen-Macaulay bipartite  graphs such as regularity, extremal Betti numbers are studied in \cite{JACM, CG19}.      
In this article, we characterize graphs whose  parity binomial edge ideals are complete intersections (Theorems \ref{bipartite-ci}, \ref{nonbipartite-ci}). We also classify graphs whose LSS ideals, permanental edge ideals and  parity binomial edge ideals are almost complete intersections. We prove that these
are either a subclass of trees, a subclass of unicyclic graphs or a subclass of bicyclic graphs(Theorems \ref{bipartite-aci}, \ref{odd-unicyclic-aci},  \ref{bicycilc-aci}, \ref{non-bipartite-odd-aci},  \ref{nonbipartite-even-aci}).

 A lot of
asymptotic invariants of an  ideal can be computed using the Rees algebra of that ideal. 
We study the Rees algebra of almost complete intersection LSS
ideals, permanental edge ideals and  parity binomial edge ideals.  
Cohen-Macaulayness of the  Rees algebra and the associated graded ring
of ideals have been a long-studied problem in commutative algebra. If
an ideal is complete intersection in a Cohen-Macaulay local
ring, then the corresponding associated graded ring and the Rees
algebra are known to be Cohen-Macaulay. In general, computing the
depth of these blowup algebras is a non-trivial problem. If an ideal
is an almost complete intersection ideal, then the
Cohen-Macaulayness of the Rees algebra and the associated graded ring
are closely related by a result of Herrmann, Ribbe and Zarzuela (see
Theorem \ref{aci-cmrees}).  To study the Cohen-Macaulayness of the associated graded ring of almost complete intersection LSS ideals, permanental edge ideals and  parity binomial edge ideals, we compute the projective dimension of almost complete intersection LSS ideals, permanental edge ideals and parity binomial edge ideals (Theorems \ref{pd-odd-unicyclic}, \ref{depth-odd-aci}, \ref{depth-even-aci}, \ref{depth-kite-aci}). We prove that the associated graded ring
and  the Rees algebra of almost complete intersection LSS ideals, permanental
edge ideals and parity binomial edge ideals are Cohen-Macaulay (Theorems \ref{rees-cohen1}, \ref{rees-cohen2}).  

An ideal $I$ of a commutative ring $A$ is said to be of \textit{linear type} if its Rees algebra  and symmetric algebra are isomorphic. In other words, the defining ideal of the Rees algebra is generated by linear forms. In general, it is quite a hard
task to describe the defining ideals of Rees algebras and  symmetric algebras.    
Huneke proved that if $I$ is generated by $d$-sequence, then $I$ is of linear type, \cite{Hu80}.  We compute the defining ideal of symmetric algebra of LSS ideals of trees and odd unicyclic graphs (Theorems \ref{syzygy-tree}, \ref{syzygy-odd-unicyclic}). In this process, we obtain second graded Betti number of LSS ideals of trees and odd unicyclic graphs (Theorems \ref{betti-tree}, \ref{betti-odd-unicyclic}).   We prove that if $L_G$ is an
almost complete intersection ideal, then $L_G$ is generated by a
$d$-sequence (Theorem \ref{d-seq-LSS}). This gives us the defining ideals of
the Rees algebras of almost complete intersection LSS ideals.

The article is organized as follows. We collect the notation and
related definitions in the second section. In Section 3, we characterize complete intersection parity binomial edge ideals. Also, we classify almost complete intersection LSS ideals, permanental edge ideals and parity binomial edge ideals. We study the Cohen-Macaulayness of Rees algebra of almost complete
intersection LSS ideals, permanental edge ideals and parity binomial edge ideals in Section 4. In Section 5, we describe the second graded 
Betti numbers and syzygies of the LSS ideals of trees and odd
unicyclic graphs. In particular, we describe the defining ideal of symmetric algebra of LSS ideals of trees and odd unicyclic graphs.

\section{Preliminaries}
Let $G$  be a  simple graph with the vertex set $[n]$ and edge set
$E(G)$. A graph on $[n]$ is said to be a \textit{complete graph}, if
$\{i,j\} \in E(G)$ for all $1 \leq i < j \leq n$. Complete graph on
$[n]$ is denoted by $K_n$. For $A \subseteq V(G)$, $G[A]$ denotes the
\textit{induced subgraph} of $G$ on the vertex set $A$, that is, for
$i, j \in A$, $\{i,j\} \in E(G[A])$ if and only if $ \{i,j\} \in
E(G)$.  For a vertex $v$, $G \setminus v$ denotes the  induced
subgraph of $G$ on the vertex set $V(G) \setminus \{v\}$.  A vertex $v \in V(G)$ is
said to be a \textit{cut vertex} if $G \setminus v$ has  more
connected components than $G$. A subset
$U$ of $V(G)$ is said to be a \textit{clique} if $G[U]$ is a complete
graph. A vertex $v$ of $G$ is said to be a \textit{simplicial vertex}
if $v$ is contained in only one maximal clique otherwise it is called an \textit{internal vertex}. For a vertex $v$,
$N_G(v) = \{u \in V(G) :  \{u,v\} \in E(G)\}$ denotes the
\textit{neighborhood} of $v$ in $G$ and  $G_v$ is the graph on the vertex set
$V(G)$ and edge set $E(G_v) =E(G) \cup \{ \{u,w\}: u,w \in N_G(v)\}$.
The \textit{degree} of a vertex  $v$, denoted by $\deg_G(v)$, is
$|N_G(v)|$.   For an edge $e$ of $G$, $G\setminus e$ is the graph
on the vertex set $V(G)$ and edge set $E(G) \setminus \{e\}$.  Let
$u,v \in V(G)$ be such that $e=\{u,v\} \notin E(G)$, then we denote by
$G_e$, the graph on the vertex set $V(G)$ and edge set $E(G_e) = E(G) \cup
\{\{x,y\} : x,\; y \in N_G(u) \; or \; x,\; y \in N_G(v) \}$.  A
\textit{cycle} is a connected graph $G$ with $\deg_G(v) = 2$ for all $v \in V(G)$. 
A cycle on $n$ vertices is denoted by $C_n$.  A  \textit{tree} is a connected graph which does not contain a
cycle. A
graph is said to be a \textit{unicyclic} graph, if it contains exactly
one cycle. The \textit{girth} of a graph $G$ is the length of a shortest
cycle in $G$. A unicyclic graph with even girth is called an \textit{even unicyclic} and with odd girth is called  an \textit{odd unicyclic} graph.   A graph $G$ is said to be \textit{bipartite} if 
there is a bipartition of $V(G)=V_1 \sqcup V_2$ such that for each
$i=1,2$, no two of the vertices of $V_i$ are adjacent, otherwise it is called \textit{non-bipartite} graph.  A complete bipartite graph on $m+n$ vertices, denoted by
$K_{m,n}$, is the graph having a vertex set $V(K_{m,n}) = \{u_1, \ldots,
u_m\} \cup \{v_1, \ldots, v_n\}$ and $E(K_{m,n}) = \{ \{u_i, v_j\} ~ :
~ 1 \leq i \leq m, 1 \leq j \leq n \}$. A \textit{claw} is the
complete bipartite graph $K_{1,3}$. A claw $\{u,v,w,z\}$ with center
$u$ is the graph with vertices $\{u,v,w,z\}$ and edges $\{\{u,v\},
\{u,w\}, \{u,z\}\}$. For a graph $G$, let $\mathfrak{C}_G$ denote the
set of all induced claws in $G$. A maximal subgraph of $G $ without a cut vertex is called a  \textit{block} of $G$. A graph $G$ is said to be a  \textit{block} graph if each  block of $G$ is a clique. If each block of a graph is either a cycle or an edge, then it is called a \textit{cactus} graph. A cactus graph such that exactly two blocks are cycles is called a \textit{bicyclic cactus} graph. Let $u,v \in V(G)$. Then $d(u,v)$ is length of a shortest path between $u$ and $v$ in $G$. A $(u,v)$-walk is a sequence of edges $ \{u,v_1\},\ldots, \{v_k,v\}$ in $G$.

Now, we recall the necessary notation from commutative algebra.
Let $A = \KK[x_1,\ldots,x_m]$ be a polynomial ring over an arbitrary
field $\KK$ and $M$ be a finitely generated graded  $A$-module. 
Let
\[
0 \longrightarrow \bigoplus_{j \in \mathbb{Z}} A(-j)^{\beta_{p,j}^A(M)} 
\overset{\phi_{p}}{\longrightarrow} \cdots \overset{\phi_1}{\longrightarrow} \bigoplus_{j \in \mathbb{Z}} A(-j)^{\beta_{0,j}^A(M)} 
\overset{\phi_0}{\longrightarrow} M\longrightarrow 0,
\]
be the minimal graded free resolution of $M$, where
$A(-j)$ is the free $A$-module of rank $1$ generated in degree $j$.
The number $\beta_{i,j}^A(M)$ is  called the {\it 
$(i,j)$-th graded Betti number} of $M$. The \textit{projective dimension} of $M$, denoted by $\pd_A(M)$, is defined as \[\pd_A(M):=\max\{i : \beta_{i,j}^A(M) \neq 0\}.\] It follows from the Auslander-Buchsbaum formula that $\depth_A(M)=m-\pd_A(M)$. We say that 
$M$ is  a \textit{finitely presented} $A$-module if there exists an
exact sequence of the form $A^p\stackrel{\varphi}{\longrightarrow} A^q
\stackrel{\psi}{\longrightarrow} M\longrightarrow 0$.  If
$q=\sum_{j \in \mathbb{Z}}\beta_{0,j}^A(M)$ and $p=\sum_{j \in \mathbb{Z}}\beta_{1,j}^A(M)$,  then this presentation is called a \textit{minimal presentation}. A homogeneous ideal $I \subset A$ is said to be \textit{complete
	intersection} if $\mu(I) = \h(I)$, where $\mu(I)$ denotes the
cardinality
of a minimal homogeneous generating set of $I$.  It is said to be 
\textit{almost complete intersection} if $\mu(I) = \h(I) + 1$ and
$I_{\mathfrak{p}}$ is complete intersection for all minimal primes
$\mathfrak{p}$ of $I$. Also, we say that $A/I$ is \textit{almost Cohen-Macaulay} if  $\depth_A(A/I)=\dim(A/I)-1$.

Let $G$ be a graph on $[n]$ and $S=\K[x_1,\ldots,x_n,y_1,\ldots,y_n]$. For an edge $e = \{i,j\}\in E(G)$ with $i <
j$, we define $f_e = f_{i,j} :=x_i y_j - x_j y_i$,  $g_e=g_{i,j}:=x_ix_j+y_iy_j$ and $\bar{g}_e=\bar{g}_{i,j}:=x_ix_j-y_iy_j$.
For  $T \subset [n]$, let $\bar{T} = [n]\setminus T$ and $c_G(T)$
denotes the number of connected components of $G[\bar{T}]$. Also, let $G_1,\ldots,G_{c_G(T)}$ be the connected 
components of $G[\bar{T}]$ and for every $i$, $\hat{G_i}$ denote the
complete graph on $V(G_i)$. Let $P_T(G) := (\underset{i\in T} \cup \{x_i,y_i\}, J_{\hat{G_1}},\ldots, J_{\hat{G}_{c_G(T)}}).$
Herzog et al. proved that $J_G$ is a radical ideal, \cite[Corollary 2.2]{HH1}. Also, they proved that for $T \subset [n]$, $P_T(G)$ is a prime ideal  and $J_G =  \underset{T \subseteq [n]}\cap
P_T(G)$, \cite[Theorem 3.2]{HH1}. 
A set $T\subset [n]$ is said to have \textit{cut point property} if for every
$i \in T$, $i$ is a cut vertex of the graph $G[\bar{T} \cup \{i\}]$.
They showed 
that $P_T(G)$ is a minimal prime
of $J_G$ if and only if either $T = \emptyset$ or $T \subset [n]$ has cut point property, \cite[Corollary 3.9]{HH1}.

We now recall some facts about LSS ideals. 
\subsection{Primary decomposition of  $L_G$ when $\sqrt{-1} \notin \KK$ and char($\K) \neq 2$}  Herzog et al. studied several properties of $L_G$. We recall some of those results which we require from   \cite{her3}:
\begin{itemize}
	\item	Set $I_{K_1}=(0)$, $I_{K_2}=(x_1x_2+y_1y_2)$. For $n>2$,  define the ideal  $I_{K_n}$ generated by
	the following binomials
	\begin{eqnarray*}
		g_{ij}&=& x_ix_j+y_iy_j, \quad 1\leq i<j\leq n,
		\nonumber
		\\
		f_{ij}&=&x_iy_j-x_jy_i, \quad 1\leq i<j\leq n,\\
		h_i&=&x_i^2+y_i^2, \quad\quad\quad  1\leq i\leq n.
		\nonumber
	\end{eqnarray*}
	
	\item
	For $1\leq m<n$  define the ideal $I_{K_{m,n-m}}$ generated by
	the  following binomials
	
	\begin{eqnarray*}
		g_{ij}&=& x_ix_j+y_iy_j, \quad 1\leq i\leq m, \quad m+1\leq j\leq n,
		\nonumber
		\\
		f_{ij}&=&x_iy_j-x_jy_i, \quad 1\leq i<j\leq m \quad \text{or}
		\quad m+1\leq i<j\leq n.
	\end{eqnarray*}
\end{itemize}
Then $I_{K_{m,n-m}}$ and $I_{K_n}$ are prime ideals, \cite[Theorems 2.4, 2.5]{her3}.
Let $G$ be  a  connected  graph on
the vertex set $V(G)=[n]$. If $G$ is non-bipartite, then we denote by
$\widetilde{G}$ the complete graph on the vertex set $V(G)$. If $G$ is a
bipartite graph, then there exists a bipartition of $V(G)=V_1 \sqcup V_2$, in this case, we denote by
$\widetilde{G}$ the complete bipartite graph on the vertex set $V(G)$ with respect
to the bipartition $V(G)=V_1 \sqcup V_2$. 

Let $G$ be a graph on the vertex set $[n]$. For
$T\subset [n]$, let $G_1,\ldots, G_{c_G(T)}$ are the connected  components of
$G[\bar{T}]$ and 
\[
Q_T(G)=(x_i,y_i : i\in T)+ I_{\widetilde{G}_1}+\ldots+ I_{\widetilde{G}_{c_G(T)}}.
\]

For $T \subset [n]$, $Q_T(G)$ is a prime ideal, \cite[Proposition 4.2]{her3}. 
Notice that if $G$ is a connected bipartite graph with bipartition $V(G)=V_1 \sqcup V_2$,
then $Q_{\emptyset}(G)=I_{K_{V_1,V_2}}$,
and if $G$ is a connected non-bipartite graph, then $Q_{\emptyset}(G)=I_{K_n}$.
For $T \subset [n]$, $b_G(T)$ is the number of bipartite connected components of
$G[\bar{T}]$. Here we consider an isolated vertex  as a bipartite graph.  For $T \subset [n]$, 	$\h (Q_T(G))=n+|T|-b_G(T),$ \cite[Proposition 4.1]{her3}. By \cite[Theorem 4.3]{her3}, we have \[L_G =\bigcap_{T \subset [n]} Q_T(G).\] 

The vertex $i\in [n]$ is said to be a
\textit{cut vertex} of $G$ if $c_G(\{i\}) >c_G(\emptyset)$ and it is said to be a
\textit{bipartition vertex} of $G$ if $b_G(\{i\}) >b_G(\emptyset)$.
Let $\mathcal{C}(G)$ be the collection of  sets $T\subset [n]$ such that each $i\in T$ is
either a cut vertex or a bipartition vertex of the graph $G[\bar{T}\cup \{i\}]$. In particular, $\emptyset\in \mathcal{C}(G)$. By \cite[Theorem 5.2]{her3}, for $T \subset [n]$, $Q_T(G)$ is a minimal prime of $L_G$ if and only if $T\in \mathcal{C}(G)$. Hence, we have
\[L_G =\bigcap_{T \in \mathcal{C}(G)} Q_T(G).\]
\subsection{Primary decomposition of  $\mathcal{I}_G$ for  any field $ \K$}
In \cite{KCT}, Kahle et al.  computed  primary decomposition of  parity binomial edge ideals of graphs. Here, we recall their results:
\begin{itemize}
	\item For a graph $G$, \begin{align*}
		W_G=&(\bar{g}_{i,j}: \text{ there is an odd }(i,j)-\text{walk in } G)\\&+(f_{i,j}: \text{ there is an even }(i,j)-\text{walk in } G).\end{align*}
	\item For a non-bipartite graph $G$, let 
	\[\mathfrak{p}^+(G)=(x_i+y_i: i \in V(G)), \; \mathfrak{p}^-(G)=(x_i-y_i: i \in V(G)).\] 
	\item For $T \subset [n]$, without loss of generality, we assume that $G_1,\ldots, G_{b_G(T)}$ are bipartite connected components of $G[\bar{T}]$ and $G_{b_G(T)+1},\ldots,G_{c_G(T)}$ are non-bipartite connected components of $G[\bar{T}]$. 
	\item For $T \subset [n]$ and $\sigma =(\sigma_{b_G(T)+1},\ldots,\sigma_{c_G(T)})\in \{+,-\}^{c_G(T)-b_G(T)}$, we associate an ideal 
	\[\mathfrak{p}_T^{\sigma}(G)=(x_i,y_i : i \in T)+ \sum_{i=1}^{b_G(T)} W_{G_i} +\sum_{j=b_G(T)+1}^{c_G(T)} \mathfrak{p}^{\sigma_j}(G_j).  \]
	 Then, $\mathfrak{p}_T^{\sigma}(G)$ is a prime ideal  \cite[Proposition 4.2]{KCT} and  \[\h(\mathfrak{p}^{\sigma}_T(G))=n+|T|-b_G(T).\]
	\item  If $\mathfrak{P}$ is a minimal prime ideal of $\mathcal{I}_G$, then $\mathfrak{P}=\mathfrak{p}_T^{\sigma}(G)$, for some $T \in \mathcal{C}(G)$ and $\sigma \in \{+,-\}^{c_G(T)-b_G(T)}$, \cite[Proposition 4.2, Lemma 4.4, Lemma 4.9]{KCT}.
	\item For $T \in \mathcal{C}(G)$, set $\mathcal{A}_T=\{t \in T: b_G(T)=b_G(T\setminus \{t\})\}$. Let $t \in \mathcal{A}_T$. We denote by $\mathcal{B}_T(t)$, the set of connected components of $G[\bar{T}]$ which are joined in $G[\bar{T} \cup \{t\}]$. Note that elements of $\mathcal{B}_T(t)$ are non-bipartite connected components of $G[\bar{T}]$.
	\item Let $T \in \mathcal{C}(G)$ and $\sigma \in  \{+,-\}^{c_G(T)-b_G(T)}$. The prime ideal $\mathfrak{p}_T^{\sigma}(G)$ is \textit{sign-split} prime ideal if for all $t \in \mathcal{A}_T$ the prime summands of $\mathfrak{p}_T^{\sigma}(G)$ corresponding to elements of $\mathcal{B}_T(t)$ has not the  same sign.
	\item A prime ideal  $\mathfrak{P}$ is a minimal prime of $\mathcal{I}_G$  if and only if $\mathfrak{P}=\mathfrak{p}_T^{\sigma}(G)$, for some sign-split prime ideal $\mathfrak{p}_T^{\sigma}(G),$ \cite[Theorem 4.15]{KCT}.
\end{itemize}
If characteristic of $\K$ is not two, then  parity binomial edge ideal of a graph is a radical ideal, \cite[Theorem 5.5]{KCT}. Hence, in this case, we have \[\mathcal{I}_G=\bigcap_{T \in \mathcal{C}(G)} \bigcap_{\sigma \in \{+,-\}^{c_G(T)-b_G(T)}} \mathfrak{p}_T^{\sigma}(G).\]

\section{(Almost)Complete Intersection  Ideals}
In this section, we classify complete intersection LSS ideals and parity binomial edge ideals. We also classify graphs whose LSS ideals and parity binomial edge ideals are  almost complete intersections.  
We first recall a fact about bipartite graphs from \cite{dav}. 
\begin{remark}\cite[Corollary 6.2]{dav}\label{main-rmk}\rm{
Let $G$ be a bipartite graph with bipartition $[n]=V_1\sqcup V_2$. We define $\varPhi_1 : S \rightarrow S$ as \[\varPhi_1(x_i)= \left\{
\begin{array}{ll}
x_i & \text{ if } i \in V_1 \\
y_i & \text{ if } i \in V_2
\end{array} \right.  \text{  and  } \;~  \varPhi_1(y_i)= \left\{
\begin{array}{ll}
y_i & \text{ if } i \in V_1 \\
-x_i & \text{ if } i \in V_2
\end{array} \right.\] and  
$\varPhi_2 : S \rightarrow S$ as \[\varPhi_2(x_i)= \left\{
\begin{array}{ll}
x_i & \text{ if } i \in V_1 \\
y_i & \text{ if } i \in V_2
\end{array} \right.  \text{  and  } \;~  \varPhi_2(y_i)= \left\{
\begin{array}{ll}
y_i & \text{ if } i \in V_1 \\
x_i & \text{ if } i \in V_2.
\end{array} \right.\] It is clear that $\varPhi_1$ and $\varPhi_2$ are isomorphism and $\varPhi_1(J_G)=L_G$ and $\varPhi_2(J_G)=\mathcal{I}_G$.  }
\end{remark}
We now begin with the classification of bipartite graphs whose LSS ideals, as well as parity binomial edge ideals, are complete intersections. In \cite{AV19}, Conca and Welker characterized graphs whose LSS ideals are complete intersections. They  proved that $L_G$ is complete intersection if and only if  $G$ does not contain claw  or even cycle (\cite[Theorem 1.4]{AV19}). Here, we give alternate form of their theorem and prove that $L_G$ is complete intersection if and only if $\mathcal{I}_G$ is complete intersection.
\begin{theorem}\label{bipartite-ci}
	Let $G$ be a bipartite graph on $[n]$. Then $L_G$ is complete intersection if and only if $\mathcal{I}_G$ is complete intersection if and only if $G$ is a disjoint union of paths. 
\end{theorem}
\begin{proof}
Since $G$ is a bipartite graph, by Remark \ref{main-rmk}, $L_G =\varPhi_1( J_G)$ and $\mathcal{I}_G =\varPhi_2( J_G)$. Therefore, $L_G$ is complete intersection if and only if $J_G$ is complete intersection if and only if $\mathcal{I}_G$ is complete intersection. Hence, the desired result follows from \cite[Theorem 1]{Rin13}.
\end{proof}
 Now, we move on to characterize non-bipartite graphs whose LSS ideals, permanental edge ideals and parity binomial edge ideals  are complete intersections. For this, we need the following lemma.
 \begin{lemma}\label{colon-non-bipartite}
 	Let $G$ be a non-bipartite graph on $[n]$. Assume that there exists  $e=\{u,v\} \in E(G)$ such that $G\setminus e$ is a bipartite graph. Then $$L_{G\setminus e}:g_e=L_{G\setminus e}+(f_{i,j}: i,j \in N_{G\setminus e}(u) \text{ or } i,j \in N_{G\setminus e}(v))=\varPhi_1(J_{(G\setminus e)_e})$$ and $$\mathcal{I}_{G \setminus e}:\bar{g_e}=\mathcal{I}_{G \setminus e}+(f_{i,j}: i,j \in N_{G\setminus e}(u) \text{ or } i,j \in N_{G\setminus e}(v))= \varPhi_2(J_{(G\setminus e)_e}).$$
 \end{lemma}
 \begin{proof}
 	Since $G$ is a non-bipartite graph and $G \setminus e$ is a bipartite graph with bipartition $[n]=V_1 \sqcup V_2$, we get that either $u,v \in V_1$ or  $u,v \in V_2$. Therefore, $\varPhi_1(g_e) = g_e$ and $\varPhi_2(\bar{g_e}) = \bar{g_e}$.  By Remark \ref{main-rmk}, we have \[L_{G \setminus e}:g_e
 	=\varPhi_1(J_{G \setminus e}) : \varPhi_1(g_e)=\varPhi_1(J_{G\setminus e}:g_e) \text{ and }\] \[\mathcal{I}_{G \setminus e}:\bar{g_e}
 	=\varPhi_2(J_{G \setminus e}) : \varPhi_2(\bar{g_e})=\varPhi_2(J_{G\setminus e}:\bar{g_e}) .\] Now, consider \[J_{G \setminus e}:g_e = \bigcap_{T \subset [n]} (P_T(G\setminus e) :g_e).\] Since generating set  of $P_T(G\setminus e)$ is a Gr$\ddot{\text{o}}$bner basis of $P_T(G\setminus e)$ with respect to lex order induced by $x_1>\cdots>x_n>y_1>\cdots> y_n$, we have $g_e = x_ux_v+y_uy_v \in P_T(G\setminus e)$ if and only if either $u \in T$ or $v\in T$ if and only if $\bar{g_e}= x_ux_v-y_uy_v \in P_T(G\setminus e)$. Therefore, 
 	\[J_{G \setminus e}:g_e = \bigcap_{T \subset [n]} (P_T(G\setminus e) :g_e)= \bigcap_{T \subset ([n] \setminus \{u,v\})} P_T(G\setminus e)=J_{(G\setminus e)_e},\] where the last equality follows from  \cite[Proposition 2.1]{RR14}.
 	Hence, $L_{G \setminus e}:g_e= \varPhi_1(J_{(G\setminus e)_e})=L_{G \setminus e}+(f_{i,j}: i,j \in N_{G\setminus e}(u) \text{ or } i,j \in N_{G\setminus e}(v))$. In a similar manner one can prove that  $\mathcal{I}_{G \setminus e}:\bar{g_e}= \varPhi_2(J_{(G\setminus e)_e})=\mathcal{I}_{G \setminus e}+(f_{i,j}: i,j \in N_{G\setminus e}(u) \text{ or } i,j \in N_{G\setminus e}(v))$.
 \end{proof}

Due to the following remark, if char$(\K)=2$, then $\Pi_G$ and $J_G$ are  essentially the same  and if char$(\K) \neq 2$, then $\Pi_G$ is essentially same as $\mathcal{I}_G$.
\begin{remark}\label{rmk-parity}\rm{
		Let $G$ be a graph with vertex set $[n]$. If char$(\K) = 2$, it follows from their definitions that $\mathcal{I}_G =L_G$ and $\Pi_G =J_G$.  Suppose char($\K)\neq 2$.  We define $\eta : S \rightarrow S$ as 
		\[\eta(x_i)=x_i +  y_i 
		\text{  and  } \;~  \eta(y_i)= 
		x_i- y_i \text{ for all } i \in V(G)
		.\] It is clear that $\eta$ is an isomorphism and $\Pi_G=\eta(\mathcal{I}_G)$. If $\sqrt{-1} \in \KK$ and char($\K) \neq 2$, then   we define $\Psi : S \rightarrow S$ as 
		\[\Psi(x_i)=x_i + \sqrt{-1} y_i 
		\text{  and  } \;~  \Psi(y_i)= 
		x_i- \sqrt{-1}y_i \text{ for all } i \in V(G)
		.\] It is clear that $\Psi$ is an isomorphism and $L_G$ is the image of permanental ideal $\Pi_G$, i.e  $\Psi(\Pi_G)=L_G$. Thus, if $\sqrt{-1} \in \K$ and char$(\K) \neq 2$, then  $\Psi(\eta(\mathcal{I}_G))=L_G$.}
\end{remark}

 \begin{theorem}\label{nonbipartite-ci}
 	Let $G$ be a connected non-bipartite graph on $[n]$. Then $L_G$ is complete intersection if and only if $G$  is an odd cycle if and only if $\mathcal{I}_G$ is complete intersection.  
 \end{theorem}
 \begin{proof}
 First, assume that $\mathcal{I}_G$ is complete intersection. Since $G$ is a non-bipartite graph, $\mathfrak{p}^+(G)$ is a minimal prime ideal of $\mathcal{I}_G$ and $\h(\mathfrak{p}^+(G))=n$.  Therefore, $\h(\mathcal{I}_G)=n=\mu(\mathcal{I}_G)$ which implies that $G$ is an odd  unicyclic graph. Let $u$ be a vertex which is part of the unique odd cycle. Since $u$ is a bipartition vertex of $G$, $\{u\} \in \mathcal{C}(G)$.  If $\deg_G(u)\geq 3$, then  $b_G(\{u\})\geq 2$. Thus,  $\h(\mathfrak{p}_{\{u\}}^{\sigma}(G))=n+1-b_G(\{u\}) \leq n-1$ and $\mathfrak{p}_{\{u\}}^{\sigma}(G)$ is a minimal prime ideal of $\mathcal{I}_G$, which conflicts the fact that $\h(\mathcal{I}_G) =n$. Consequently, $\deg_G(u) =2$ and hence, $G$ is an odd cycle. 
 
 Now, we assume that $L_G$ is complete intersection and char$(\K) \neq 2$. If $\sqrt{-1} \in \K$, then by Remark \ref{rmk-parity}, $\mathcal{I}_G$ is complete intersection and hence, $G$ is an odd cycle.  Suppose $\sqrt{-1} \notin \K$, then 
  $Q_{\emptyset}(G)=I_{K_n}$ is a minimal prime of $L_G$ as $G$ is a non-bipartite graph.  It follows from  \cite[Proposition 2.3]{her3} that $\h(I_{K_n})=n$.  Therefore, $\h(L_G)=n=\mu(L_G)$ which implies that $G$ is an odd  unicyclic graph. If $u$ is a vertex  of the unique odd cycle, then $u$ is a bipartition vertex of $G$. Thus, $\{u\} \in \mathcal{C}(G)$. Now, if $\deg_G(u)\geq 3$, then  $b_G(\{u\})\geq 2$ and hence,  $\h(Q_{\{u\}}(G)) \leq n-1$, which is a contradiction. This implies that $\deg_G(u) = 2$. Hence, $G$ is an odd cycle. 
  
  Conversely, we have to prove that $L_{C_n}$ and $\mathcal{I}_{C_n}$ are  complete intersections, for $n$ odd. Let $e=\{1,n\}$, then $C_n \setminus e=P_n$. By Theorem \ref{bipartite-ci}, $L_{P_n}$ and $\mathcal{I}_{P_n}$ are complete intersections. Note that $L_{C_n}=L_{P_n}+(g_e)$ and $\mathcal{I}_{C_n}=\mathcal{I}_{P_n}+(\bar{g_e})$. Therefore, it is enough to prove that  $L_{P_n}:g_{e}=L_{P_n}$ and $\mathcal{I}_{P_n}:\bar{g_{e}}=\mathcal{I}_{P_n}$ which immediately follows from Lemma \ref{colon-non-bipartite}.  Hence, $L_{C_n}$ and $ \mathcal{I}_{C_n}$  are complete intersections.
   \end{proof}
   
It follows from \cite[Corollary 4.6]{her3} that if $G$ is a graph on $n$ vertices, then $\h(L_G) \leq n-b_G(\emptyset)$. As a consequence we have the following:
\begin{corollary}\label{graph-ci}
	Let $G$ be a graph on $[n]$. Then $L_G$ is complete intersection if and only if $\mathcal{I}_G$ is complete intersection if and only if  all the bipartite connected components of $G$ are paths and non-bipartite connected components are odd cycles.
\end{corollary}
   Now, we move on to find  connected  graphs whose LSS ideals and  parity binomial edge ideals are almost complete intersections. In \cite{JAR}, Jayanthan et al. characterized connected graphs whose binomial edge ideals are  almost complete intersections.
\begin{theorem}\label{bipartite-aci}
	Let $G$ be a connected bipartite graph on $[n]$ which is not a path. Then $L_G$ is an almost complete intersection ideal if and only if $\mathcal{I}_G$ is   almost complete intersection  if and only if $G$ is either obtained by adding an edge between two disjoints paths or by adding an edge between two vertices of a path  such that the girth of $G$ is even.
\end{theorem}
\begin{proof}
Since $G$ is a  bipartite graph, by Remark \ref{main-rmk}, $L_G =\varPhi_1(J_G)$ and $\mathcal{I}_G =\varPhi_2(J_G)$. Hence, the proof follows from \cite[Theorems 4.3, 4.4]{JAR}. 	
\end{proof}  
For $A \subseteq [n]$ and $i \in A$, we define
$p_A(i) = |\{j \in A \mid j \leq i\}|$.  We now give complete classification of odd unicyclic graphs whose LSS ideals and  parity binomial edge ideals are  almost complete intersections. 
\begin{theorem}\label{odd-unicyclic-aci}
	Let $G$ be a connected odd unicyclic graph on $[n]$. Assume that char$(\K)\neq 2$. Then $L_G$ is an almost complete intersection ideal if and only if $\mathcal{I}_G$ is  almost complete intersection  if and only if $G$ is one of the following types:
	\begin{enumerate}
		\item $G$ is obtained by adding an edge $e$ between an odd cycle and a path,
		\item $G$ is obtained by adding an edge $e$ between two vertices of a path such that girth of $G$ is odd and at least one of the vertex  is an  internal vertex of the path, 
		\item $G$ is obtained by attaching a path of length  $\geq 1$ to each vertex of a triangle.
	\end{enumerate}
\end{theorem}
\begin{proof}
First, assume that $\mathcal{I}_G$ is an almost complete intersection ideal. Therefore, $\h(\mathcal{I}_G)=\mu(\mathcal{I}_G)-1=n-1$.    We claim that $\deg_G(u)\leq 3$, for every $u \in V(G)$. Let if possible, there exist a vertex $u$ such that  $\deg_G(u) \geq 4$. Then $\{u\} \in \mathcal{C}(G)$ and $b_G(\{u\}) \geq 3$. Therefore, $\h(\p_{\{u\}}^{\sigma}(G))=n+1-b_G(\{u\})\leq n-2$, which is a contradiction. Hence, $\deg_G(u) \leq 3$ for all $u \in V(G)$.  Now, let $u,v \in V(G)$ be distinct vertices such that  $\deg_G(u) =3$ and $\deg_G(v)=3$.
If $\{u,v\} \notin E(G)$, then for $T=\{u,v\} $,  $T \in \mathcal{C}(G)$ and $b_G(T) = 4$. Consequently,
$\h(\p_T^{\sigma}(G))= n+|T|-b_G(T) = n-2$ which conflicts the fact that $\h(\mathcal{I}_G)=n-1$. 	
Thus, if
two vertices  have degree three, then they are adjacent. If the 
number of vertices of  degree three is at most $2$, then $G$ is either of type $(1)$ or type $(2)$. If  the number of vertices of degree three is more than two, then the odd cycle has length $3$, each vertex of the cycle has degree   three and these are only vertices with degree three.   Hence, $G$ is of type $(3)$. 
	
Now, we assume that $L_G$ is an almost complete intersection ideal. Therefore, $\h(L_G)=\mu(L_G)-1=n-1$.  Suppose $\sqrt{-1} \in \K$, then by Remark \ref{rmk-parity}, $\mathcal{I}_G$ is  almost complete intersection and hence, we are done.  Suppose $\sqrt{-1} \notin \K$,  then we claim that $\deg_G(u)\leq 3$, for every $u \in V(G)$. If not, then  there is a vertex $u$ such that  $\deg_G(u) \geq 4$. Clearly, $\{u\} \in \mathcal{C}(G)$ and $b_G(\{u\}) \geq 3$. Therefore, $\h(Q_{\{u\}}(G))=n+1-b_G(\{u\})\leq n-2$, which conflicts the fact that $\h(L_G)=n-1$. Hence $\deg_G(u) \leq 3$.  Now, let  $u ,v \in V(G)$  such that  $\deg_G(u) =3$, $\deg_G(v)=3$ and $u \neq v$.
If $\{u,v\} \notin E(G)$, then for $T=\{u,v\} $,  $T \in \mathcal{C}(G)$ and $b_G(T) =4$. Therefore,
$\h(Q_T(G))= n+|T|-b_G(T) = n-2$ which is a contradiction. 	
Now, the proof is in the same lines as the proof for $\mathcal{I}_G$.

Conversely, if $G$ is either of type $(1)$ or of type $(2)$, then $L_G = L_{G \setminus e}+ (g_e)$  and $\mathcal{I}_G= \mathcal{I}_{G\setminus e}+(\bar{g_e})$. By Corollary \ref{graph-ci}, $L_{G \setminus e}$ and $\mathcal{I}_{G\setminus e}$ are complete intersections.  Since char$(\K) \neq 2$,  $L_{G\setminus e}$ and $\mathcal{I}_{G\setminus e}$ are radical ideal. Therefore, $L_{G \setminus e}:g_e = L_{G \setminus e} : g_e^2$ and $\mathcal{I}_{G \setminus e}:\bar{g_e} = \mathcal{I}_{G \setminus e} : \bar{g_e}^2$.  Hence, by \cite[Theorem 4.7(ii)]{HMV89} and Theorem \ref{nonbipartite-ci}, the assertion follows.

Now, we assume that $G$ is of type $(3)$. Let $u,v,w $ be the vertices of the cycle and $T \in \mathcal{C}(G)$. We claim that  $\{u,v,w\} \cap T= \emptyset$ if and only if $\h(\p_T^{\sigma}(G))=n$. First assume that $\h(\p_T^{\sigma}(G))=n$. Since $\h(\p_T^{\sigma}(G))=n+|T|-b_G(T)=n$, we have $|T|=b_G(T)$. If possible, let $\{u,v,w\} \cap T \neq \emptyset$. Without loss of generality, we may assume that $u \in T$. One can note that $T\setminus \{u\} \in \mathcal{C}(G\setminus u)$ and $b_{G\setminus u}(T\setminus \{u\})=b_G(T)$. Since $G\setminus u$ is disjoint union of two paths, $\mathcal{I}_{G\setminus u}$ is complete intersection  and therefore, $\h(\mathcal{I}_{G\setminus u})=n-3=\h(\p_{T\setminus \{u\}}^{\sigma}(G\setminus u))=n-1+|T\setminus \{u\}|-b_{G\setminus u}({T\setminus \{u\}})$. Thus, we have $|T|=b_{G\setminus u}(T\setminus \{u\})-1=b_G(T)-1$, which is a contradiction.   Conversely, if $\{u,v,w\} \cap T=\emptyset $ and $T \neq \emptyset$, then every element of $T$ has degree two in $G$ and for every pair $u',v' \in T$, $\{u',v'\} \notin E(G)$. Thus, by deleting each of the elements of $T$ increases the number of bipartite connected components of the corresponding graph by one and hence, $b_G(T)=|T|$. From the proof of the claim, we observe that $\h(\mathcal{I}_G)=n-1=\mu(\mathcal{I}_G)-1$. Let $T \in \mathcal{C}(G)$ such that $\h(\p_T^{\sigma}(G)) = n-1$. Then $\{u,v,w\} \cap T \neq \emptyset$ and $\{u,v,w\} \not\subset T$. We may assume that $u \in T$. Let $N_G(u)=\{v,w,z\}$. Since $T \in \mathcal{C}(G)$, $z \notin T$. Note that $A=\{u,v,w,z\}$ forms a claw in $G$ with center $u$ and  $(-1)^{p_A(v)}f_{z,w}\bar{g}_{{u,v}} + (-1)^{p_A(z)}f_{v,w}\bar{g}_{{u,z}} + (-1)^{p_A(w)}f_{v,z}\bar{g}_{{u,w}}=0.$ 
The minimal presentation of $\mathcal{I}_G$ is \[S^{\beta_2(S/\mathcal{I}_G)} \stackrel{\phi}{\longrightarrow} S^n \longrightarrow \mathcal{I}_G \longrightarrow 0 .\] Therefore, $f_{z,w}, f_{v,w}, f_{v,z} \in I_1(\phi)$, where $I_1(\phi)$ is an ideal generated by entries of the matrix $\phi$. Since $z \notin T$ and $\{u,v,w\} \not\subset T$, we have  $I_1(\phi) \not\subset \p_T^{\sigma}(G)$. Consequently, by \cite[Lemma 1.4.8]{bh}, $\mu((\mathcal{I}_G)_{\p_T^{\sigma}(G)}) \leq n-1$. As $\h((\mathcal{I}_G)_{\p_T^{\sigma}(G)}) =n-1$, by \cite[Theorem 13.5]{hm}, $\mu((\mathcal{I}_G)_{\p_T^{\sigma}(G)})\geq n-1$. Hence, $(\mathcal{I}_G)_{\p_T^{\sigma}(G)}$ is  complete intersection. Now, if $T \in \mathcal{C}(G)$ such that $\h(\p_T^{\sigma}(G)) =n$, then  it follows from  \cite[Theorem 13.5]{hm}, $\mu((\mathcal{I}_G)_{\p_T^{\sigma}(G)})\geq n
$. Since $\mu((\mathcal{I}_G)_{\p_T^{\sigma}(G)}) \leq \mu(\mathcal{I}_G)=n$, we have  $\mu((\mathcal{I}_G)_{\p_T^{\sigma}(G)})= n$. Hence, $(\mathcal{I}_G)_{\p_T^{\sigma}(G)}$ is  complete intersection.	

Suppose  $\sqrt{-1} \in \K$, then  $\mathcal{I}_G$ is  almost complete intersection and hence, $L_G$ is almost complete intersection, by Remark \ref{rmk-parity}. It remains to  prove that if $G$ is of type $(3)$ and  $\sqrt{-1}  \notin \K$, then $L_G$ is an almost complete intersection ideal. The proof is in the same lines as the proof for $\mathcal{I}_G$ by replacing $\mathfrak{p}^{\sigma}_T(G)$ by $Q_T(G)$.	
\end{proof}	
One can observe that if $G$ is a connected non-bipartite graph, then $\p^{+}(G)$ is a minimal prime of $\mathcal{I}_G$. Therefore, $\h(\mathcal{I}_G)\leq \h(\p^{+}(G))=n$. If $\sqrt{-1} \notin \K$, char$(\K) \neq 2$ and $G$ is a non-bipartite graph, then  $I_{K_n}$ is one of the minimal primes of $L_G$. Therefore, $\h(L_G)\leq \h(I_{K_n})=n$. If $G$ is a connected non-bipartite graph such that $L_G$ or $\mathcal{I}_G$ is almost complete intersection, then $n \leq |E(G)| \leq n+1$. We now assume that $G$ is connected non-bipartite graph other than odd unicyclic graph, i.e. $|E(G)| = n+1$. So, $G$ is  obtained by  adding an edge in a unicyclic graph. First, we give  classification of a connected non-bipartite bicyclic cactus graph whose LSS ideals and  parity binomial edge ideals are almost complete intersections.
\begin{theorem}\label{bicycilc-aci}
Let $G$ be a connected non-bipartite bicyclic cactus graph on $[n]$. Assume that  char$(\K) \neq  2$.	Then $L_G$ is almost complete intersection if and only if $\mathcal{I}_G$ is almost complete intersection if and only if $G$ is obtained by adding an edge $e$ between two disjoint odd cycles.	
\end{theorem}
\begin{proof}
First, assume that $\mathcal{I}_G$ is almost complete intersection. Therefore, $\h(\mathcal{I}_G)=\mu(\mathcal{I}_G)-1=n$.	We claim that the distance between the two cycles is $\geq 1$. If both cycles share a common vertex say $v$, then $\{v\} \in \mathcal{C}(G)$, $b_G(\{v\}) \geq 2$ and $\h(\p_{\{v\}}^{\sigma}(G)) \leq n-1$, which is not possible as $\h(\mathcal{I}_G)=n$. Now, we claim that both cycles of $G$ are odd cycles.  Let $u$ be the vertex of an odd cycle and $v$ be the vertex of  another cycle such that $d(u,v)$ is the distance between the two cycles. Clearly, $\{u\} \in \mathcal{C}(G)$. If $v$ is the vertex of an even cycle, then  $b_G(\{u\}) \geq 2$. So,  $\h(\p_{\{u\}}^{\sigma}(G))\leq n-1$ which is not possible. Thus, both  cycles of $G$ are odd cycles. If $d(u,v)\geq 2$, then $T=\{u,v\} \in \mathcal{C}(G)$ and $b_G(T) \geq 3$. Consequently, $\p_T^{\sigma}(G)$ is a minimal prime of $ \mathcal{I}_G$ and $\h(\p_T^{\sigma}(G)) \leq n-1$, which conflicts the fact that $\h(\mathcal{I}_G)=n$. Hence, $\{u,v\} \in E(G)$. Let $w \in V(G)\setminus \{u,v\}$. One can note that either $\{u,w\} \notin E(G)$ or $\{v,w\} \notin E(G)$. If $\deg_G(w) \geq 3$, then either $T=\{u,w\} \in \mathcal{C}(G)$ or $T=\{v,w\} \in \mathcal{C}(G)$. In either case $b_G(T) \geq 3$ which implies that $\h(\p_T^{\sigma}(G)) \leq n-1$. Therefore, $\deg_G(w) =2$, for $w \in V(G)\setminus \{u,v\}$. Hence, $G$ is obtained by adding an edge $e$ between two disjoint odd cycles.		
	
Now, assume that $L_G$ is almost complete intersection.  Suppose $\sqrt{-1} \in \K$, then by Remark \ref{rmk-parity}, $\mathcal{I}_G$ is  almost complete intersection and hence, $G$ satisfies the hypothesis.  Suppose $\sqrt{-1} \notin \K$, then the proof is in the same lines as the proof for $\mathcal{I}_G$.

Conversely, if $G$ is obtained by adding an edge $e$ between two disjoint odd cycles, then $L_G = L_{G \setminus e}+ (g_e)$ and $\mathcal{I}_G = \mathcal{I}_{G \setminus e}+ (\bar{g_e})$. By virtue of Corollary \ref{graph-ci}, both $L_{G \setminus e}$ and $\mathcal{I}_{G\setminus e}$ are complete intersections. As char$(\K) \neq 2$,  $L_{G\setminus e}$ and $\mathcal{I}_{G\setminus e}$ are radical ideal. Therefore, $L_{G \setminus e}:g_e = L_{G \setminus e} : g_e^2$ and $\mathcal{I}_{G \setminus e}:\bar{g_e} = \mathcal{I}_{G \setminus e} : \bar{g_e}^2$.  Hence, by \cite[Theorem 4.7(ii)]{HMV89} and Theorem \ref{nonbipartite-ci}, the assertion follows.
\end{proof}	
We now  consider the case when $G$ is a connected non-bipartite  graph on $[n]$ with $|E(G)|=n+1$ and it is  not a bicyclic cactus graph. Therefore, $G$ is obtained from a unicyclic graph $H$ on $m$ vertices by attaching a path $P_{n-m+2}$ between two distinct vertices of the unique cycle of $H$. More preciesly, let $H$ be a unicyclic graph on $m$ vertices and $u,v$  distinct vertices of the unique cycle of $H$. Let $G$ be the graph obtained from $H$ by attaching one end vertex  of $P_{n-m+2}$ at $u$ and another end vertex at $v$.  Note that $T=\{u,v\} \in \mathcal{C}(G)$. If $n-m>0$, then $b_G(T) \geq 3$ and therefore, $\h(\p_T^{\sigma}(G)) \leq n-1$ and $\h(Q_T(G)) \leq n-1$. Also, if $\deg_G(u) \geq 4$, then $u$ is a cut vertex of $G$ with $b_G(\{u\}) \geq 2$. So, $\{u\} \in \mathcal{C}(G)$ and $\h(\p_{\{u\}}^{\sigma}(G)) \leq n-1$, $\h(Q_{\{u\}}(G)) \leq n-1$. Similarily, if $\deg_G(v) \geq 4$, then $\{v\} \in \mathcal{C}(G)$ and  $\h(\p_{\{v\}}^{\sigma}(G)) \leq n-1$, $\h(Q_{\{v\}}(G)) \leq n-1$. Thus, if $\mathcal{I}_G$ or $L_G$ is almost complete intersection, then $n=m$ and $\deg_G(u) =\deg_G(v)=3$, i.e. $G$ is obtained by adding a chord $e =\{u,v\}$ in a unicyclic graph $H$ on $[n]$ such that $\deg_H(u)=\deg_H(v)=2$.
\begin{theorem}\label{non-bipartite-odd-aci}
Let $G$ be a connected graph which is obtained by adding a chord $e=\{u,v\}$ in an odd unicyclic graph $H$ such that $\deg_H(u)=\deg_H(v)=2$. Assume that  char$(\K) \neq  2$. Then $L_G$ is almost complete intersection if and only if   $\mathcal{I}_G$ is almost complete intersection if and only if $H$ is  an odd cycle.
\end{theorem}
\begin{proof}
First, we assume that $\mathcal{I}_G$ is almost complete intersection. Consequently, $\h(\mathcal{I}_G)=\mu(\mathcal{I}_G)-1 =n$. If $w \notin \{u,v\}$ is a vertex of an induced odd cycle of $G$ such that $\deg_G(w) \geq 3$, then  $\{w\} \in \mathcal{C}(G)$, $b_G(\{w\})\geq 2$ and hence,  $\h(\mathcal{I}_G)\leq \h(\p_{\{w\}}^{\sigma}(G))\leq n -1$. Now, if $w \notin \{u,v\}$ is a vertex of an even induced cycle such that $\deg_G(w) \geq 3$, then either $\{u,w\} \notin E(G)$ or $\{v,w\} \notin E(G)$. Assume that $\{u,w\} \notin E(G)$. Therefore,  $T= \{u,w\} \in \mathcal{C}(G)$, $b_G(T) \geq 3$ and hence, $\p_T^{\sigma}(G)$ is a minimal prime of  $\mathcal{I}_G$ with $\h(\mathcal{I}_G)\leq \h(\p_T^{\sigma}(G))\leq n -1$. We have a contradiction in each case. Thus, $\deg_G(w)=2$, for $w \in V(G) \setminus \{u,v\}$. 	Hence, $H$ is an odd cycle.	
	
Now, assume that $L_G$ is almost complete intersection.  Suppose $\sqrt{-1} \in \K$, then by Remark \ref{rmk-parity}, $\mathcal{I}_G$ is  almost complete intersection and hence, $H$ is an odd cycle.  Suppose $\sqrt{-1} \notin \K$, then the proof is in the same lines as the proof for $\mathcal{I}_G$.

 Conversely, if $G$ is obtained by adding a chord $e$  in an odd cycle $H$, then $L_G = L_{H}+ (g_e)$ and $\mathcal{I}_G = \mathcal{I}_{H}+ (\bar{g_e})$. By Theorem \ref{nonbipartite-ci}, $L_H$ and $\mathcal{I}_H$ are complete intersections. Since char$(\K) \neq 2$,  $L_{H}$ and $\mathcal{I}_{H}$ are radical ideal. Therefore, $L_{H}:g_e = L_{H} : g_e^2$ and $\mathcal{I}_{H}:\bar{g_e} = \mathcal{I}_{H} : \bar{g_e}^2$.  Hence, the assertion follows from \cite[Theorem 4.7(ii)]{HMV89} and Corollary \ref{graph-ci}. 
\end{proof}	
\begin{theorem}\label{nonbipartite-even-aci}
	Let $G$ be a connected non-bipartite graph which  is obtained by adding a chord $e=\{u,v\}$ in an even unicyclic graph $H$ such that $\deg_H(u)=\deg_H(v)=2$. Then $L_G$ is almost complete intersection  if and only if $\mathcal{I}_G$ is almost complete intersection  if and only if $H$ is one of the following:
	\begin{enumerate}
		\item $H$ is  an even cycle, 
		\item $H$ is obtained by attaching  a path to a vertex $i$ of an even cycle such that $\{u,i\}, \{v,i\}$ are edges of the even cycle.
	\end{enumerate}
\end{theorem}
\begin{proof}
First, assume that $\mathcal{I}_G$ is an almost complete intersection ideal. Therefore, $\h(\mathcal{I}_G)=\mu(\mathcal{I}_G)-1 =n$.  Note that $G$ has two induced  odd cycles. 
If $w$ is not a vertex of an induced cycle and $\deg_G(w) \geq 3$, then  $\{w\} \in \mathcal{C}(G)$, $b_G(\{w\})\geq 2$ and $ \h(\p_{\{w\}}^{\sigma}(G))\leq n -1$ which is a  contradiction to the fact that $\h(\mathcal{I}_G) =n$.  Therefore, $\deg_G(w) \leq 2$, if $w$ is not a vertex of an induced cycle. Now, we assume that $w$ is a  vertex of an induced odd cycle. We claim that $\deg_G(w) \leq 3$. Suppose that $\deg_G(w)\geq 4$, then $w$ is a bipartition vertex of $G$. Consequently,  $\{w\} \in \mathcal{C}(G)$, $b_G(\{w\}) \geq 2$ and $\h(\p_{\{w\}}^{\sigma}(G)) \leq n-1$, which is a contradiction. Hence, $\deg_G(w) \leq 3$, if $w$ is a vertex of an induced cycle. If $\deg_G(w)=2$, for $w \in V(G) \setminus \{u,v\}$, then $H$ is an even cycle. Now, let $w$ be a vertex of the cycle other than $u$ and $v$ such that $\deg_G(w)=3$. If either $\{u,w\} \notin E(G)$ or $\{v,w\} \notin E(G)$ then, for $T= \{u,w\} \text{ or } \{v,w\}$ respectively, $\p_T^{\sigma}(G)$ is a minimal prime of  $\mathcal{I}_G$ with $ \h(\p_T^{\sigma}(G))\leq n -1$ as $b_G(T)\geq 3$. Thus, if $w \in V(G) \setminus \{u,v\}$ such that $\deg_G(w)=3$, then $\{u,w\},\{v,w\} \in E(G)$. Now, if $w,w' \in V(G) \setminus \{u,v\}$ such that $\deg_G(w)=\deg_G(w')=3$ and $w \neq w'$, then $T=\{w,w'\} \in \mathcal{C}(G)$ and $b_G(T) \geq 3$, which conflicts the fact that $\h(\mathcal{I}_G) =n$. Therefore, $H$ is obtained  by attaching  a path to a vertex $i$ of an even cycle such that $\{u,i\}, \{v,i\}$ are edges of the even cycle.
	
Now, assume that $L_G$ is almost complete intersection and char$(\K) \neq 2$.  Suppose $\sqrt{-1} \in \K$, then by Remark \ref{rmk-parity}, $\mathcal{I}_G$ is  almost complete intersection and hence, $H$ is of the  required type.  Suppose $\sqrt{-1} \notin \K$, then the proof is in the same lines as the proof for $\mathcal{I}_G$.

We now prove the converse. Assume that $H$ is an even cycle. First, we prove that $\h(\mathcal{I}_G) =n$. Let $T \in \mathcal{C}(G)$ such that $T \neq \emptyset$. We claim that, if $u \in T$ or $v \in T$, then $|T|=b_G(T)$. We can assume that $u \in T$. Clearly, $T \setminus \{u\} \in \mathcal{C}(G\setminus u)$. One can note that $G \setminus u$ is path graph on $n-1$ vertices. Therefore, $L_{G \setminus u}$ and $ \mathcal{I}_{G\setminus u}$ are complete intersections and hence $\h(\mathcal{I}_{G\setminus u})=\h(L_{G \setminus u}) =n-2=n-1+|T\setminus \{u\}|-b_{G\setminus u}(T\setminus \{u\})$. Consequently, we have $|T|=b_G(T)$. Now, assume that $\{u,v\} \cap T =\emptyset$. Let $w \in T$. Then $T\setminus w \in \mathcal{C}(G\setminus w)$. Observe that $G\setminus w$ is an odd unicyclic graph such that $\mathcal{I}_{G\setminus w} $ is either complete intersection or almost complete intersection. If $\mathcal{I}_{G\setminus w}$ is complete intersection, then $\h(\mathfrak{p}^{\sigma}_{T\setminus \{w\}}(G\setminus w)) =n-1+|T\setminus \{w\}|-b_{G\setminus w}(T\setminus \{w\})= n-1$. Consequently, we have $|T|=b_G(T)+1$.  If $\mathcal{I}_{G\setminus w}$ is almost complete intersection, then $G\setminus w$ is of type $(2)$ graph in Theorem \ref{odd-unicyclic-aci}. Since each vertex of $T\setminus w$ has degree two in $G\setminus w$, deleting each of the elements of $T\setminus w$ increases the number of bipartite connected components of the corresponding graph by one. Therefore, $|T\setminus \{w\}|=b_{G\setminus w}(T\setminus \{w\})=b_G(T)$ which further implies that $|T|=b_G(T)+1$. Thus,   $\h(\p_T^{\sigma}(G)) =n+1$ if and only if $T \neq \emptyset$ and $T \cap \{u,v\} =\emptyset$.   By \cite[Theorem 13.5]{hm},  $(\mathcal{I}_G)_{\p_T^{\sigma}(G)}$ is complete intersection ideal, if $\h(\p_T^{\sigma}(G))=n+1$. Now, let $T \in \mathcal{C}(G)$ such that $\h(\p_T^{\sigma}(G))=n$. The minimal presentation of $\mathcal{I}_G$ is \[S^{\beta_2(S/\mathcal{I}_G)} \stackrel{\varphi}{\longrightarrow} S^{n+1} \longrightarrow \mathcal{I}_G \longrightarrow 0 .\]Let $Y=y_1\cdots y_n$.  Define $b \in S^n$ as follows:
\[
(b)_k = \frac{Y}{y_ky_{k+1}} \text{ for }1 \leq k \leq n-1,
(b)_n = \frac{Y}{y_1y_n}. \] It follows from  \cite[Theorem 3.5]{JAR} that $\sum_{k=1}^{n-1}
(b)_kf_{{k,k+1}} - (b)_nf_{{1,n}} =0.$ Consequently, we have  $\sum_{k=1}^{n-1}
\varPhi_2((b)_k)\bar{g}_{{k,k+1}} - \varPhi_2((b)_n)\bar{g}_{{1,n}} =0.$ Therefore, $\varPhi_2((b)_k) \in I_1(\varphi)$, for $1 \leq k \leq n$. If $T = \emptyset$, then $\varPhi_2((b)_k) \notin \mathfrak{p}^{\sigma}_{\emptyset}(G)$ which implies that $I_1(\varphi) \not\subset \mathfrak{p}^{\sigma}_{\emptyset}(G)$. We now consider that $T \neq \emptyset$. Since $\h(\mathfrak{p}^{\sigma}_{T}(G))=n$, we have $\{u,v\} \cap T \neq \emptyset$. Without loss of generality assume that $u \in T$. As $\deg_G(u)=3$, let $N_G(u)=\{v,w,z\}$. Note that $T\setminus \{u\} \in \mathcal{C}(G \setminus u)$. Therefore, $w, z \notin T$. Notice that $A=\{u,v,w,z\}$ forms a claw in $G$ with center $u$ and  $(-1)^{p_A(v)}f_{z,w}\bar{g}_{{u,v}} + (-1)^{p_A(z)}f_{v,w}\bar{g}_{{u,z}} + (-1)^{p_A(w)}f_{v,z}\bar{g}_{{u,w}}=0.$ Consequently, $f_{z,w}, f_{v,w}, f_{v,z} \in I_1(\varphi)$. If $z$ and $w$ belong to different components of $G[\bar{T}]$, then $f_{z,w} \notin \mathfrak{p}^{\sigma}_T(G)$. In the case of $z$ and $w$ belongs to same component of $G[\bar{T}]$,  then $v$ and $z$ belong to different partition of bipartite graph $G\setminus u$. Therefore, $f_{v,z} \notin \mathfrak{p}^{\sigma}_{T}(G)$. Thus, $I_1(\varphi) \not\subset \mathfrak{p}^{\sigma}_{T}(G)$, if $\h(\mathfrak{p}^{\sigma}_{T}(G)) =n$.   By virtue of \cite[Lemma 1.4.8]{bh}, $\mu((\mathcal{I}_G)_{\mathfrak{p}^{\sigma}_{T}(G)}) \leq n$, if $\h(\mathfrak{p}^{\sigma}_{T}(G))=n$. Hence, it follows from \cite[Theorem 13.5]{hm} that $\mathcal{I}_G$ is almost complete intersection.

We now assume that $H$ satisfies hypothesis $(2)$. Let $T \in \mathcal{C}(G)$ such that $T \neq \emptyset$. If $u \in T$ or $v \in T$, then following the proof of type $(1)$, $|T|=b_G(T)$ . So, assume that $\{u,v\} \cap T =\emptyset$. Let $w \in V(G) \setminus \{u,v\}$ be such that $\deg_G(w) =3$. If $w \in T$, then $T \setminus \{w\} \in \mathcal{C}(G\setminus w)$ and $b_G(T)=b_{G\setminus w}(T\setminus \{w\})$. Since $G\setminus w$ is disjoint union of a path and an odd cycle, $\mathcal{I}_{G\setminus w}$ is complete intersection. Therefore, $\h(\mathfrak{p}^{\sigma}_{T\setminus \{w\}}(G\setminus w)) =n-2+|T\setminus \{w\}|-b_G(T)=n-3$ which further implies that $b_G(T) =|T|$. Now, we consider that $\{u,v,w\} \cap T = \emptyset$. In the case that $T$ contains a vertex of cycle say $z$, then $T\setminus \{z\} \in \mathcal{C}(G\setminus z)$ and $\mathcal{I}_{G\setminus z}$ is almost complete intersection of type $(3)$ in Theorem \ref{odd-unicyclic-aci}. As we have proved in Theorem \ref{odd-unicyclic-aci}, $\h(\mathfrak{p}^{\sigma}_{T\setminus \{z\}}(G\setminus z))=n-1+|T\setminus \{z\}|-b_{G\setminus z}(T\setminus \{z\})=n-1$ which implies that $|T|=b_G(T)+1$.  If none of the vertices of the cycle belongs to $T$, then  removing each of the elements of $T$ increases the number of bipartite connected components of the corresponding graph by  one. Consequently, we have $|T| = b_G(T)$. Thus, $\h(\mathcal{I}_G)=n$. Also, $\h(\mathfrak{p}^{\sigma}_T(G))=n+1$ if and only if $T \neq \emptyset$, $T \cap \{u,v,w\} = \emptyset$ and  $T$ contains at least one vertex of cycle.  Following the proof for type $(1)$, one can prove that $(\mathcal{I}_G)_{\mathfrak{p}^{\sigma}_T(G)}$ is complete intersection for all $T \in \mathcal{C}(G)$. Hence, $\mathcal{I}_G$ is an almost complete intersection ideal.

Suppose  $\sqrt{-1} \in \K$ and char$(\K) \neq 2$, then  $L_G$ is  almost complete intersection, by Remark \ref{rmk-parity} and the above paragraph. Now, it remains to  prove that if $H$ is either even cycle or $H$ satisfies hypothesis $(2)$, $\sqrt{-1}  \notin \K$ and char$(\K) \neq 2$, then $L_G$ is an almost complete intersection ideal. The proof of which is in the same lines as the proof for $\mathcal{I}_G$ by replacing $\mathfrak{p}^{\sigma}_T(G)$ by $Q_T(G)$.	
\end{proof}

We conclude this section by characterizing disconnected graphs whose LSS ideals and parity binomial edge ideals are  almost complete intersections.
\begin{corollary}
Let $G=G_1 \sqcup \cdots \sqcup G_k$ be a disconnected graph on $[n]$. Then $\mathcal{I}_G$ is almost complete intersection if and only if $L_G$ is almost complete intersection if and only if for some $i$, $\mathcal{I}_{G_i} $ is almost complete intersection and for $j \neq i$, $\mathcal{I}_{G_j}$ are complete intersections. 
\end{corollary}

\section{Cohen-Macaulayness of the Rees Algebra}
Let $G$ be a simple graph on $[n]$ and $R=S[T_{\{i,j\}} ~ : ~\{i,j\} \in E(G) \text{ with } i<j]$.  
Let $\delta, \gamma : R \to S[t]$ be given by
$\delta(T_{\{i,j\}}) = g_{i,j}t, \; \gamma(T_{\{i,j\}}) = \bar{g}_{i,j}t$. Then  Im$(\delta) = \R(L_G)$, Im$(\gamma) = \R(\mathcal{I}_G)$ and $\ker (\delta)$, $\ker(\gamma)$ are
called the defining ideals of $\R(L_G)$, $\R(\mathcal{I}_G)$ respectively. 
We now study the Cohen-Macaulayness of the Rees algebra of almost complete intersection LSS ideals and parity binomial edge ideals.  We first recall a result
that characterizes the Cohen-Macaulayness of the Rees algebra and the
associated graded ring of an almost complete intersection ideal.
\begin{theorem}\cite[Corollary 1.8]{Herr}\label{aci-cmrees}
	Let $A$ be a Cohen-Macaulay local (graded) ring and $I \subset A$ be an
	 almost complete intersection (homogeneous) ideal in $A$. Then 
	\begin{enumerate}
		\item $\gr_A(I)$ is Cohen-Macaulay if and only if
		$\depth(A/I) \geq \dim(A/I) -1.$
		\item $\R(I)$ is Cohen-Macaulay if and only if $\h(I) > 0$ and
		$\gr_A(I)$ is Cohen-Macaulay.
	\end{enumerate}
\end{theorem}
Thus,  to prove that  $\R(\mathcal{I}_G)$ is
Cohen-Macaulay, it is enough to prove that $ \depth(S/\mathcal{I}_G) \geq
\dim(S/\mathcal{I}_G)-1,$ which is equivalent to prove that $S/\mathcal{I}_G$ is either Cohen-Macaulay or almost Cohen-Macaulay. Similarly,  to prove that  $\R(L_G)$ is
Cohen-Macaulay, it is enough to prove that  $S/L_G$ is either Cohen-Macaulay or almost Cohen-Macaulay. 

\begin{lemma}\label{betti-lemma}
Let $G$ be a graph on $[n]$. Then $\beta_{i,j}(S/L_G) =\beta_{i,j}(S/\mathcal{I}_G)$ for all $i,j$. In particular, $\pd(S/L_G)=\pd(S/\mathcal{I}_G)$, $\dim(S/L_G)=\dim(S/\mathcal{I}_G)$ and $\depth(S/L_G)= \depth(S/\mathcal{I}_G)$.
\end{lemma}
\begin{proof}
If char$(\K) =2$, then $L_G = \mathcal{I}_G$ and hence, we are done. Assume now that char$(\K) \neq 2$.
If $\sqrt{-1} \in \K$, then the assertion follows from Remark \ref{rmk-parity}. Suppose that $\sqrt{-1} \notin \K$ and set $\mathbb{L}=\K(\sqrt{-1})$, $S'=\mathbb{L} \otimes_{\K} S$. Let $(\mathcal{F}_{{\cdot}}, d^{\mathcal{F}}_{\cdot})$ and $(\mathcal{G}_{\cdot}, d^{\mathcal{G}}_{\cdot} )$ be  minimal free resolution of $S/L_G$ and $S/\mathcal{I}_G$, respectively. Since $\K \subset \mathbb{L}$ is faithfully flat extension, $(\mathbb{L}\otimes_{\K} \mathcal{F}_{\cdot}, \textbf{1}_{\mathbb{L}} \otimes_{\K} d^{\mathcal{F}}_{\cdot})$ and $(\mathbb{L}\otimes_{\K} \mathcal{G}_{\cdot},\textbf{1}_{\mathbb{L}} \otimes_{\K} d^{\mathcal{G}}_{\cdot})$ are   free resolutions of $S'/L_G$ and $S'/\mathcal{I}_G$ respectively. Since for each $i$, $\textbf{1}_{\mathbb{L}} \otimes_{\K} d^{\mathcal{F}}_{i}  =d^{\mathcal{F}}_{i}$ and $\textbf{1}_{\mathbb{L}} \otimes_{\K} d^{\mathcal{G}}_{i}  =d^{\mathcal{G}}_{i}$, $(\mathbb{L}\otimes_{\K} \mathcal{F}_{\cdot}, \textbf{1}_{\mathbb{L}} \otimes_{\K} d^{\mathcal{F}}_{\cdot})$ and $(\mathbb{L}\otimes_{\K} \mathcal{G}_{\cdot},\textbf{1}_{\mathbb{L}} \otimes_{\K} d^{\mathcal{G}}_{\cdot})$ are minimal   free resolution of $S'/L_G$ and $S'/\mathcal{I}_G$ respectively. Consequently, $\beta_{i,j}^S(S/L_G) =\beta_{i,j}^{S'}(S'/L_G) =\beta_{i,j}^{S'}(S'/\mathcal{I}_G)=\beta_{i,j}^S(S/\mathcal{I}_G)$.
\end{proof}   
Due to Lemma \ref{betti-lemma}, it is enough to study the Cohen-Macaulayness of almost complete intersection parity binomial edge ideals.

The following fundamental property of projective dimension is used repeatedly in this section.

\begin{lemma}\label{pd-lemma}
	Let $S$ be a standard graded polynomial ring. Let $M,N$ and $P$ be finitely generated graded $S$-modules. 
	If $ 0 \rightarrow M \xrightarrow{f}  N \xrightarrow{g} P \rightarrow 0$ is a 
	short exact sequence with $f,g$  
	graded homomorphisms of degree zero, then 
	\begin{enumerate}
		\item[(i)] $\pd_S(M) \leq \max \{\pd_S(N), \pd_S(P)-1\},$
		\item[(ii)] $\pd_S(P) \leq \max \{\pd_S(N), \pd_S(M)+1\}$,
		\item[(iii)] $\pd_S(P) = \pd_S(N)$ if  $\pd_S(N) > \pd_S(M)$.
	\end{enumerate}	
\end{lemma}
It follows from  \cite[Theorems 4.3, 4.7]{JAR} that if $G$ is a tree, then  $\mathcal{I}_G$ is almost complete intersection ideal if and only if $S/\mathcal{I}_G$ is almost Cohen-Macaulay. Consequently, $\R(\mathcal{I}_G)$ and $\gr_S(\mathcal{I}_G)$ are Cohen-Macaulay. Now, we prove the same for odd unicyclic graphs. First, we compute the  projective dimension of parity binomial edge ideal of an odd unicyclic graph.
\begin{theorem}\label{pd-odd-unicyclic}
	Let $G$  be a connected odd unicyclic graph on $[n]$. Then $\pd(S/\mathcal{I}_G)=n$.
\end{theorem}
\begin{proof}
	Since $\mathfrak{p}^{+}(G)$ is a minimal prime of $\mathcal{I}_G$, we get  that $\pd(S/\mathcal{I}_G) \geq \h(\mathfrak{p}^+(G))=n$. Let $e=\{u,v\}$ be an edge of the cycle. Now, consider the short exact sequence 
	\begin{equation}\label{ses-LS}
	0 \longrightarrow \frac{S}{\mathcal{I}_{G\setminus e}:\bar{g}_e}(-2) \stackrel{\cdot \bar{g}_e}\longrightarrow \frac{S}{\mathcal{I}_{G\setminus e}} \longrightarrow \frac{S}{\mathcal{I}_G} \longrightarrow 0.
	\end{equation}  Observe that $G \setminus e$ is a tree and  $(G\setminus e)_e$ is a block graph on $[n]$. It follows from \cite[Theorem 1.1]{her1} that $\pd(S/J_{G\setminus e}) =n-1$ and $\pd(S/J_{(G\setminus e)_e}) =n-1$. Therefore, by virtue of Lemma \ref{colon-non-bipartite}, $\pd(S/\mathcal{I}_{G \setminus e}:\bar{g}_e)=\pd(S/J_{(G\setminus e)_e}) =n-1$ and by Remark \ref{main-rmk}, $\pd(S/\mathcal{I}_{G\setminus e}) =\pd(S/J_{G\setminus e}) =n-1$. Hence, by applying Lemma \ref{pd-lemma} on the  short exact sequence \eqref{ses-LS},  $\pd(S/\mathcal{I}_G) \leq n$.
\end{proof}
\begin{theorem}\label{rees-cohen1}
	Let $G$ be a connected  odd unicyclic graph on $[n]$.  Assume that  char$(\K) \neq 2$. Then the following are equivalent:
	\begin{enumerate}
		\item $S/\mathcal{I}_G$ is almost Cohen-Macaulay,
		\item $\mathcal{I}_G$ is almost complete intersection.
	\end{enumerate}
	In particular, $\R(\mathcal{I}_G)$ and $\gr_S(\mathcal{I}_G)$ are Cohen-Macaulay, if $\mathcal{I}_G$ is almost complete intersection.
	\begin{proof}
		By  Auslander-Buchsbaum formula and Theorem \ref{pd-odd-unicyclic}, $\depth(S/\mathcal{I}_G)=n$. Therefore, $S/\mathcal{I}_G$ is almost Cohen-Macaulay if and only if $\dim(S/\mathcal{I}_G)=n+1$ if and only if $\h(\mathcal{I}_G)=n-1$ if and only if $\mathcal{I}_G$ is almost complete intersection, by Theorem \ref{odd-unicyclic-aci}. 
	\end{proof}	
\end{theorem}

\begin{remark}\label{depth-rmk}{\em
	Let $G$ be a connected   even unicyclic graph such that $\mathcal{I}_G$ is an almost complete intersection ideal. Then, it follows from \cite[Lemma 4.6]{JAR} and Remark \ref{main-rmk} that $\pd(S/\mathcal{I}_G) \leq n$. By virtue of  \cite[Theorem 6.1]{dav}, $S/\mathcal{I}_G$ is Cohen-Macaulay if and  only if $G$ is obtained by attaching  a path of length $\geq 1$ to two adjacent vertices of $C_4$. Since $\dim(S/\mathcal{I}_G) =n+1$,  if $S/\mathcal{I}_G$ is not Cohen-Macaulay, then $\depth(S/\mathcal{I}_G)=n$ and hence, $S/\mathcal{I}_G$ is almost Cohen-Macaulay. Moreover,  $\gr_S(\mathcal{I}_G)$ and $\R(\mathcal{I}_G)$ are Cohen-Macaulay.}
\end{remark}

Now, we move on to study the Cohen-Macaulayness of $\R(\mathcal{I}_G)$, where $G$ is obtained by adding a chord in a unicyclic graph such that $\mathcal{I}_G$ is almost complete intersection. To do that, we need to compute the depth of $S/\mathcal{I}_G$. 

A graph $G$ is said to be \textit{closed}, if generating set  of $J_G$ is a  Gr$\ddot{\text{o}}$bner basis with respect to lexicographic order induced by $x_1 > \cdots > x_n >y_1>\cdots>y_n$. Let $H$ be a connected closed graph on $[n]$ such that $S/J_H$ is Cohen-Macaulay. By \cite[Theorem 3.1]{her1},
there exist integers $1=a_1 < a_2< \cdots < a_s <a_{s+1} = n$ such that for $1\leq i \leq s$, $F_i=[a_i,a_{i+1}]$ is a maximal clique and if $F$ is a maximal clique, then $F=F_i$ for some $1 \leq i \leq s$. Set $e=\{1,n\}$. The graph $G=H \cup \{e\}$ 
is called the \textit{quasi-cycle} graph associated to $H$. In \cite{FM}, Mohammadi and Sharifan have studied the Hilbert series of binomial edge ideal of quasi-cycles.
\begin{remark}\cite[Remark 3.3]{AR3}\label{quasi-rmk}
	Let $G$ be the quasi-cycle graph associated with a Cohen-Macaulay closed graph $H$. 
	Let  $F_1,\ldots,F_s$ be a leaf order on $\Delta(H)$. Let $\iv(G)$ denote the number of internal vertices in $G$. If $H \neq P_3$, then $\iv(G)\geq s$ 
	and $\iv(H) =s-1$.
\end{remark}

\begin{theorem}\label{quasi-block}
	Let $H$ be a connected closed graph on $[n]$ such that $S/J_H$ is Cohen-Macaulay and $G=H \cup \{e\}$ be a quasi-cycle graph associated to $H$.  Then $\pd(S/J_G) \leq n$. Moreover, if $H \neq P_3$, then $\pd(S/J_G) =n$. 
\end{theorem} 
\begin{proof}
	If $H = P_3$, then $G=K_3$ and  the result follows from \cite[Theorem 1.1]{her1}. We now assume that $H \neq P_3$.  We proceed by induction on $\iv(G)$. By virtue of Remark \ref{quasi-rmk}, $\iv(G)\geq 2$.
	If $\iv(G)=2$,  then $H=G\setminus e$ is a  block graph with exactly one internal vertex. Let $v \in V(H)$ be the 
	internal vertex of $H$. Therefore, $v$ is also an internal vertex of $G$. By \cite[Lemma 4.8]{oh}, $J_G=J_{G_v} \cap ((x_v,y_v)+J_{G\setminus v})$.
	Note that  $G_v$ is a complete graph on $[n]$ and $G\setminus v$  is a block graph on $n-1$ vertices.    Therefore, by \cite[Theorem 3.1]{her1},  $\pd(S/J_{G_v})=n-1$, 
	$\pd(S/((x_v,y_v)+J_{G \setminus v}))=n $. Note that $J_{G_v} +((x_v,y_v)+J_{G\setminus v}) = (x_v,y_v)+J_{G_v \setminus v}$.  
	Therefore, we have the following short exact sequence:
	\begin{equation}\label{ses1}
	0  \longrightarrow  \dfrac{S}{J_{G} } \longrightarrow
	\dfrac{S}{J_{G_v}} \oplus \dfrac{S}{(x_v,y_v)+J_{G \setminus v}} \longrightarrow \dfrac{S}{(x_v,y_v)+J_{G_v \setminus v} } \longrightarrow 0.
	\end{equation} Observe that $G_v \setminus v$ is a complete graph on $n-1$ vertices. Consequently, by \cite[Theorem 3.1]{her1}, $\pd(S/((x_v,y_v)+J_{G_v \setminus v}))=n $. Thus, by Lemma \ref{pd-lemma} and the short exact sequence (\ref{ses1}), $\pd(S/J_G)\leq n$.
	
	Now assume that $\iv(G)>2$.  Let  $v \in V(H)$ be an internal vertex of $H$. Therefore, $v$ is 
	an internal vertex of $G$. Notice that $G\setminus v$ is a connected Cohen-Macaulay closed graph on $n-1$ vertices, therefore by \cite[Theorem 3.1]{her1},  $\pd(S/((x_v,y_v)+J_{G\setminus v})) =n$. Also, observe that $G_v$ is a quasi-cycle graph with $\iv(G_v)=\iv(G)-1$, 
	hence, by induction $\pd(S/J_{G_v}) \leq n$.  Since $G_v \setminus v$ is a quasi-cycle on $n-1$ vertices with $\iv(G_v \setminus v)= \iv(G)-1$, by induction,  $\pd(S/((x_v,y_v)+J_{G_v \setminus v})) \leq n+1$. Hence, using Lemma \ref{pd-lemma} in the short exact sequence (\ref{ses1}), we conclude that  $\pd(S/J_G) \leq n$. Now, if $H \neq P_3$, then either $s \geq 3$ or for some $1 \leq i \leq s$, $|F_i| >2$. In first case  $T = \{a_1,a_3\}$ has the cut point property and in second case $T=\{a_i,a_{i+1}\}$ has the cut point property. In both the cases $c_G(T)=2$, consequently, $\h(P_T(G))=n+|T|-c_G(T)=n$. Hence, $\pd(S/J_G) \geq \h(P_T(G)) =n$.
\end{proof}
It follows from \cite[Corollary 4.2]{FM} and Theorem \ref{quasi-block} that if $H \neq P_3$, then $S/J_G$ is almost Cohen-Macaulay.
\begin{lemma}\label{quasi-path}
	Let $H$ be a connected closed graph on $[n-m+1]$ such that $S/J_H$ is Cohen-Macaulay and $G'=H \cup \{e'\}$ be a quasi-cycle graph associated to $H$. Let $v$ be an internal vertex of $H$ and $G$ be a graph on $[n]$ obtained by attaching a path $P_m$ to the vertex $v$ of  $G'$.  Then $\pd(S/J_G) \leq n$. Moreover, if $\iv(G')=2$, then $\pd(S/J_G)=n-1$.
\end{lemma}
\begin{proof}
	If $H=P_3$, then $G$ is a closed graph such that $S/J_G$ is Cohen-Macaulay. Thus, $\pd(S/J_G)=n-1$. Assume that $H\neq P_3$. Let $e=\{u,v\} \in E(G)$ such that $G \setminus e$ is the disjoint union of a  path $P_{m-1}$ and a quasi-cycle graph $G'$. Assume that $u \in V(P_{m-1})$. It follows from \cite[Theorem 3.4]{FM} that $J_{G\setminus e} :f_e=J_{(G\setminus e)_e}$.	Observe that $(G\setminus e)_e$ is the disjoint union of a path $P_{m-1}$ and  $G'_v$. Note that $G'_v$ is either a quasi-cycle or a complete graph. Therefore, $\pd(S/J_{(G\setminus e)_e}) =m-2+\pd(S/J_{G'_v})\leq n-1$, by Theorem \ref{quasi-block}. Also, $\pd(S/J_{G\setminus e})=m-2+\pd(S/J_{G'}) =n-1$. From the following exact sequence:
	\begin{align}\label{main-ses}
	0\longrightarrow \frac{S}{J_{G\setminus e}:f_e}(-2) \stackrel{\cdot f_e}{\longrightarrow} \frac{S}{J_{G\setminus e}} \longrightarrow \frac{S}{J_G} \longrightarrow 0 ,
	\end{align} we get,  $\pd(S/J_G) \leq n$. Now, if $\iv(G') =2$, then $(G\setminus e)_e$ is the disjoint union of a path $P_{m-1}$ and a complete graph on $n-m+1$ vertices. Consequently, $\pd(S/J_{(G\setminus e)_e})=n-2$. Hence, by Lemma \ref{pd-lemma} and the short exact sequence \eqref{main-ses}, $\pd(S/J_G)=n-1$.
\end{proof}
We now consider the case that $G$ is obtained by adding a chord in a unicyclic graph and $\mathcal{I}_G$ is almost complete intersection. Let $G$ be a graph  and $H$ a subgraph of $G$. Then $G$ is said to be $H$-free graph if $H$ is not an induced subgraph of $G$.
\begin{theorem}\label{depth-odd-aci}
	Let $G$ be a graph obtained by adding a chord $e'=\{u,v\}$ in  an odd cycle $C_n$. Then $\pd(S/\mathcal{I}_G)=n+1$, if $G$ is $C_4$-free and $\pd(S/\mathcal{I}_G)=n$, if $C_4$ is an induced subgraph of $G$. 
\end{theorem}
\begin{proof}
	Let $e=\{v,w\} \in E(G)$ be an edge of the induced odd cycle. Observe that $G\setminus e$ is an even unicyclic graph such that $\mathcal{I}_{G\setminus e}$ is almost complete intersection, by Theorem \ref{bipartite-aci} and $S/\mathcal{I}_{G\setminus e}$ is not Cohen-Macaulay, by \cite[Theorem 6.1]{dav}.  By Remark \ref{depth-rmk}, $\pd(S/\mathcal{I}_{G\setminus e})=n$. It follows from Lemma \ref{colon-non-bipartite} that $\mathcal{I}_{G \setminus e} :\bar{g}_e=\varPhi_2(J_{(G\setminus e)_e})$. Notice that $(G\setminus e)_e =(G\setminus e)_v$ is a graph obtained by attaching a path to an internal vertex of a quasi-cycle graph $G'$. If $C_4$ is an induced subgraph of $G$, then $\iv(G') =2$ and  hence,   $\pd(S/J_{(G\setminus e)_e}) =n-1$, by Lemma \ref{quasi-path}.   Now, by the short exact sequence \eqref{ses-LS} and Lemma \ref{pd-lemma}, $\pd(S/\mathcal{I}_G)=n$. In the case that $G$ is a $C_4$-free graph,  the induced even cycle has length $\geq 6$. Let $i,j \notin \{u,v\}$ be vertices of the induced even cycle such that $i$ is not adjacent to $j$.  Clearly, $T=\{i,j\} \in \mathcal{C}(G)$ and $b_G(T)=1$.  Consequently, $\pd(S/\mathcal{I}_G) \geq \h(\mathfrak{p}^{\sigma}_T(G))=n+1$. By Lemma \ref{quasi-path},  $\pd(S/J_{(G\setminus e)_e}) \leq n$. Thus, by applying Lemma \ref{pd-lemma} on the short exact sequence \eqref{ses-LS}, we get,  $\pd(S/\mathcal{I}_G) \leq n+1$,  which proves the assertion.
\end{proof}

\begin{theorem}\label{depth-even-aci}
	Let $G$ be a non-bipartite graph obtained by adding a chord $e=\{u,v\}$ in  an even cycle $C_n$. Then $\pd(S/\mathcal{I}_G)=n+1$. 
\end{theorem}
\begin{proof}
	By virtue of Lemma \ref{colon-non-bipartite}, $ \mathcal{I}_{G\setminus e} :\bar{g}_e =\varPhi_2(J_{(G\setminus e)_e})$. Observe that $(G\setminus e)_e$ is a quasi-cycle graph on $n$ vertices which is not a triangle. Therefore, $\pd(S/\mathcal{I}_{G\setminus e}:\bar{g}_e) =\pd(S/J_{(G\setminus e)_e}) =n,$  by Theorem \ref{quasi-block}. Also, $G\setminus e$ is a quasi-cycle graph on $[n]$ so that by Theorem \ref{quasi-block}, $\pd(S/\mathcal{I}_{G\setminus e})= n$. Thus, using Lemma \ref{pd-lemma} on the short exact sequence \eqref{ses-LS}, we have $\pd(S/\mathcal{I}_G) \leq n+1$. Observe that $G$ has two induced odd cycles. Let $i,j \notin V(G)\setminus \{u,v\}$  such that $i$  and $j$ are vertices of distinct induced odd cycles in $G$. Then  $T=\{i,j\} \in \mathcal{C}(G)$ and $b_G(T)=1$. Hence,  $\pd(S/\mathcal{I}_G) \geq \h(\mathfrak{p}^{\sigma}_T(G)) = n+1$ which completes   the proof.
\end{proof}
 A graph $G$ on $[5]$ with edge set $E(G)=\{\{1,2\},\{2,3\},\{3,4\},\{1,4\},\{2,4\},\{3,5\}\}$ is called \textit{Kite} graph.
\begin{theorem}\label{depth-kite-aci}
	Let $G$ be a non-bipartite graph on $[n]$. Assume that $G$ satifies hypothesis of Theorem \ref{nonbipartite-even-aci}$(2)$. Then $\pd(S/\mathcal{I}_G)= n+1$, if $G$ is a Kite-free graph and $\pd(S/\mathcal{I}_G)= n$, if Kite is an induced subgraph of $G$. 
\end{theorem}
\begin{proof}
	Note that $G \setminus e$  is an even unicyclic graph.  By Theorem \ref{bipartite-aci}, $\mathcal{I}_{G\setminus e}$ is almost complete intersection and by \cite[Theorem 6.1]{dav}, $S/\mathcal{I}_{G\setminus e}$ is not Cohen-Macaulay.  Therefore, it follows from Remark \ref{depth-rmk} that $\pd(S/\mathcal{I}_{G\setminus e}) =n$. By virtue of Lemma \ref{colon-non-bipartite}, $\mathcal{I}_{G \setminus e}:\bar{g}_e =\varPhi_2(J_{(G\setminus e)_e})$. Notice that $(G\setminus e)_e$ is a graph obtained by attaching a path to an internal vertex of a quasi-cycle graph $G'$. Now, if both the induced odd cycles of $G$ have girth three, then $\iv(G') =2$.  Thus, by Lemma  \ref{quasi-path}, $\pd(S/\mathcal{I}_{G\setminus e}:\bar{g}_e)= \pd(S/J_{(G \setminus e)_e})= n-1$. Hence, by the short exact sequence \eqref{ses-LS}, $\pd(S/\mathcal{I}_G) =n$. If $G$ has an  induced odd cycle of girth $\geq 5$, then by the  proof of Theorem \ref{nonbipartite-even-aci}(2), there exists $T \in \mathcal{C}(G)$ such that $\h(\mathfrak{p}^{\sigma}_T(G))=n+1$. Therefore, $\pd(S/\mathcal{I}_G) \geq \h(\mathfrak{p}^{\sigma}_T(G))=n+1$. By virtue of Lemmas \ref{colon-non-bipartite}, \ref{quasi-path}, $\pd(S/\mathcal{I}_{G\setminus e}:\bar{g}_e)=\pd(S/J_{(G\setminus e)_e}) \leq n$. Hence, by Lemma \ref{pd-lemma} and the short exact sequence \eqref{ses-LS}, the desired result follows.
\end{proof}

The following theorem is an immediate consequence of Theorems \ref{non-bipartite-odd-aci}, \ref{nonbipartite-even-aci}, \ref{depth-odd-aci}, \ref{depth-even-aci} and Theorem \ref{depth-kite-aci}.
\begin{theorem}\label{rees-cohen2}
	Let $G$ be a graph on $[n]$ obtained by adding a chord in a unicyclic graph. Assume that  char$(\K)\neq  2$ and $\mathcal{I}_G$ is almost complete intersection. Then
	 \begin{enumerate}
		\item $S/\mathcal{I}_G$ is Cohen-Macaulay if and only if  either $C_4$ is an induced  subgraph of $G$ or  Kite graph is an induced subgraph of $G$.
		\item $S/\mathcal{I}_G$ is almost Cohen-Macaulay if and only if  $G$ is $C_4$-free and Kite-free.
	\end{enumerate}
	Moreover, $\R(\mathcal{I}_G)$ and $\gr_S(\mathcal{I}_G)$ are Cohen-Macaulay.
\end{theorem}

\section{First Syzygy of LSS ideals}
In this section,  we compute the defining ideal of symmetric algebra of LSS ideals of trees and odd unicyclic graphs. Let $A$ be a Noetherian ring and $I \subset A$ be an ideal.  Let $A^m \stackrel{\phi}\longrightarrow A^n \longrightarrow I \longrightarrow 0$ be a presentation of $I$ and $T=[T_1 \cdots T_n]$ be a $1 \times n$ matrix of variables over ring $A$. Then the defining ideal of symmetric algebra of $I$, denoted by $Sym(I)$, is generated by entries of the matrix $T\phi$. Thus, to compute the defining ideal of symmetric algebra of LSS ideals of trees and odd unicyclic graphs, we compute the  first
syzygy of LSS ideals of trees and odd unicyclic graphs. The second graded Betti numbers of binomial edge ideals of trees are computed in \cite[Theorem 3.1]{JAR}. The results from \cite{JAR} and Remark \ref{main-rmk} gives us the  second graded Betti number LSS ideals of trees.
\begin{theorem}\label{betti-tree}
Let $G$ be a tree on $[n]$. Then  \[\beta_2(S/L_G)=	  \beta_{2,4}(S/L_G)={n-1 \choose 2}+\sum_{v \in V(G)} {\deg_G(v) \choose 3}.\]
\end{theorem}
We now describe the first syzygy of $LSS$ ideals of trees.    
\begin{theorem}\label{syzygy-tree}
Let $G$ be a tree on $[n]$. Let $\left\{e_{\{i,j\}} ~ : ~ \{i, j\} \in E(G) \right\}$ be the standard basis of $S^{n-1}$. Then the first syzygy of $L_G$
is minimally generated by elements of the form
\begin{enumerate}
  \item[(a)] $g_{i,j}e_{\{k,l\}} - g_{k,l}e_{\{i,j\}}, \text{ where }
	\{i,j\}\neq \{k,l\}\in E(G) \text{ and }$
  \item[(b)]  $(-1)^{p_A(j)}f_{k,l}e_{\{i,j\}} + (-1)^{p_A(k)}
	f_{j,l}e_{\{i,k\}} + (-1)^{p_A(l)}f_{j,k}e_{\{i,l\}}, \\ \text{ where
	} A = \{i,j,k,l\} \in \mathfrak{C}_G \text{ with center at } i.$
\end{enumerate}
\end{theorem}
\begin{proof}
The proof follows from \cite[Theorem 3.2]{JAR} and Remark \ref{main-rmk}.
\end{proof}

We now compute the second graded Betti number of LSS ideals of odd unicyclic graphs. 

\begin{theorem}\label{betti-odd-unicyclic}
Let $G$ be an odd unicyclic graph on $[n]$. Then \[\beta_2(S/L_G) = \beta_{2,4}(S/L_G)=\binom{n}{2}+ \sum_{i \in [n]} \binom{\deg_G(i)}{3}.\]
\end{theorem}
\begin{proof}
 Let $e=\{u,v\} \in E(G)$ such that $e$ is an edge of the cycle. One can note that $G\setminus e$ is a tree. We consider the  following short exact sequence:
 \begin{equation}\label{ses-LSS}
 	0 \longrightarrow \frac{S}{L_{G\setminus e}:{g}_e}(-2) \stackrel{\cdot {g}_e}\longrightarrow \frac{S}{L_{G\setminus e}} \longrightarrow \frac{S}{L_G} \longrightarrow 0.
 \end{equation} 
 The long exact sequence of Tor corresponding to the short exact sequence \eqref{ses-LSS} is 
\begin{equation}\label{Tor-les}
\cdots	\rightarrow \Tor_{2,j}^S\left(\frac{S}{L_{G\setminus e}},\K\right)\rightarrow \Tor_{2,j}^S\left(\frac{S}{L_{G}},\K\right) 
\rightarrow \Tor_{1,j}^S\left(\frac{S}{L_{G\setminus e}:g_e}(-2),\K\right)\rightarrow\cdots
\end{equation} Note that \[\Tor_{1,j}^S\left(\frac{S}{L_{G\setminus e}:g_e}(-2),\K\right) \simeq \Tor_{1,j-2}^S\left(\frac{S}{L_{G\setminus e}:g_e},\K\right).\] It follows from Lemma \ref{colon-non-bipartite} that $\beta_{1,j-2}(S/L_{G\setminus e}:g_e)=0$, if $j \neq 4$ and \[\beta_{1,2}(S/L_{G\setminus e}:g_e) =n-1+ \binom{\deg_G(u)-1}{2}+ \binom{\deg_G(v)-1}{2}.\] Now, by virtue of Theorem \ref{betti-tree}, \[\beta_2(S/L_{G\setminus e})=	  \beta_{2,4}(S/L_{G\setminus e})={n-1 \choose 2}+\sum_{i \in [n]} {\deg_{G \setminus e}(i) \choose 3}.\]
Therefore, $\beta_{2,j}(S/L_G)=0$, for $j \neq 4$. Since $\beta_{2,2}(S/L_{G\setminus e}:g_e)=0$ and $\beta_{1,4}(S/L_{G\setminus e})=0$, by \eqref{Tor-les}, $\beta_{2,4}(S/L_G)=\beta_{2,4}(S/L_{G\setminus e})+\beta_{1,2}(S/L_{G\setminus e}:g_e)$. Hence, the desired result follows.
\end{proof}
We now compute the minimal generators of the first syzygy of LSS ideals of odd unicyclic graphs.

\vskip 2mm \noindent
\textbf{Mapping Cone Construction:} 
Let $(\mathbf{F}.,d^{\mathbf{F}}.)$ and
$(\mathbf{G}.,d^{\mathbf{G}}.)$ be minimal $S$-free resolutions of
${S}/{L_{G\setminus e}}$ and $[{S}/{L_{G\setminus e}:g_e}](-2)$
respectively. Let $\varphi. :(\mathbf{G}.,d^{\mathbf{G}}.) \longrightarrow (\mathbf{F}.,d^{\mathbf{F}}.)
 $ be the
complex morphism induced by the multiplication by $g_e$. 
The mapping cone $(\mathbf{M}(\varphi).,\delta.)$ is the $S$-free resolution of
$S/L_G$
such that $(\mathbf{M}(\varphi))_i=\mathbf{F}_i\oplus
\mathbf{G}_{i-1}$ and the differential maps are
$\delta_i(x,y)=(d^{\mathbf{F}}_i(x)+\varphi_{i-1}(y),-d^{\mathbf{G}}_{i-1}(y))$
for $x\in \mathbf{F}_i$ and $y\in \mathbf{G}_{i-1}$. 
The mapping cone need not necessarily be a minimal free
resolution. We refer the reader to \cite{eisenbud95} for more details on the
mapping cone.

 \begin{theorem}\label{syzygy-odd-unicyclic}
Let $G$ be an odd unicyclic graph on $[n]$.
Let $\{e_{\{i,j\}} ~: ~ \{i, j\} \in E(G) \}$ denote the standard basis of $S^{n}$.	
Then the first syzygy of $L_G$ is minimally generated by elements of
the form
\begin{enumerate}
  \item[(a)] $g_{i,j}e_{\{k,l\}} - g_{k,l}e_{\{i,j\}}$, where $
 	\{i,j\}\neq \{k,l\}\in E(G)$
  \item[(b)] $(-1)^{p_A(v)}f_{z,w}e_{\{u,v\}} + (-1)^{p_A(z)}f_{v,w}e_{\{u,z\}} + (-1)^{p_A(w)}f_{v,z}e_{\{u,w\}},$
	where
 	$A=\{u,v,w,z\}$ forms a claw in $G$ with center $u$.
 	\end{enumerate}
 \end{theorem}

\begin{proof}
From Theorem \ref{betti-odd-unicyclic}, we know that the minimal presentation of $L_G$ is of the form
 \[  S^{\beta_{2,4}(S/L_G)} \longrightarrow
 S^{n} \longrightarrow L_{G}\longrightarrow
 0, \]
 where
 \[
   \beta_{2,4}(S/L_G)= 
	 {n \choose 2}+\sum_{v \in V(G)} {\deg_G(v)
	 \choose 3}=\binom{n}{2}+|\mathfrak{C}_G|.\]
Let $e=\{u,v\} \in E(G)$ such that $e$ is an edge of the unique odd cyclic. Since $G \setminus e$ is a tree, by Theorem \ref{syzygy-tree}, we get a minimal generating set
of the first syzygy of $L_{G\setminus e}$ as 

 \noindent
\begin{enumerate}
  \item[(a)] $g_{i,j}e_{\{k,l\}} - g_{k,l}e_{\{i,j\}}$, where $\{i,j\}\neq \{k,l\}\in E(G \setminus e),$ 
  \item[(b)] $(-1)^{p_A(j)}f_{k,l}e_{\{i,j\}} + (-1)^{p_A(k)}f_{j,l}e_{\{i,k\}} + (-1)^{p_A(l)}f_{j,k}e_{\{i,l\}},$
	where
 	$A=\{i,j,k,l\}$ forms a claw in $G\setminus e$ with center at $i$,
\end{enumerate}
By virtue of Lemma \ref{colon-non-bipartite}, we have \[L_{G\setminus e}:g_e=L_{G\setminus e}+(f_{i,j}: i,j \in N_{G\setminus e}(u) \text{ or } i,j \in N_{G\setminus e}(v)).\]
Now we apply the mapping cone construction to the short exact sequence
(\ref{ses-LSS}). Let $(\mathbf{G}.,d^{\mathbf{G}}.)$ and $(\mathbf{F}.,d^{\mathbf{F}}.)$ be minimal
free resolutions of $[S/L_{G\setminus e} :g_e](-2)$ and $S/L_{G\setminus e}$
respectively. Then $G_1 \simeq S^{\beta_{1,2}(S/L_{G\setminus e }:g_e)}, F_1 \simeq S^{n-1}$ and
$F_2\simeq S^{\beta_{2}(S/L_{G\setminus e})}$. 
Denote the standard basis of $G_1$ by $\mathcal{S}=\mathcal{S}_1\sqcup
\mathcal{S}_2 \sqcup \mathcal{S}_3$, where
$\mathcal{S}_1= \{E_{\{i,j\}} : \{i,j\}\in E(G \setminus e) \}$, $\mathcal{S}_2= \{E_{\{k,l\}} : k,l \in
N_G(v)\setminus \{u\}\}$ and $\mathcal{S}_3=\{E_{\{k,l\}} : k,l \in N_G(u)\setminus \{v\}\}$. Note that $|\mS_1| = n-1$, $|\mS_2| = {\deg_G(v)
- 1 \choose 2}$ and $|\mS_3|= \binom{\deg_G(u)-1}{2}$.  
One can note that 
\[
\begin{array}{ll}
d_1^{\mathbf{G}}(E_{\{i,j\}
})=g_{i,j}, & \text{ if } E_{\{i,j\}} \in \mathcal{S}_1,\\
d_1^{\mathbf{G}}(E_{\{i,j\}
}) =f_{i,j},  & \text{ if } E_{\{i,j\}} \in \mathcal{S}_2\sqcup \mathcal{S}_3.
\end{array}
\]
 Also, let
$\{e_{\{i,j\}} : \{i,j\}\in E(G \setminus e) \}$
be the standard basis of $F_1$. By the mapping cone construction, the
map from $G_0$ to $F_0$ is given by the multiplication by $g_e$. Now we
define $\varphi_1$ from $G_1$ to $F_1$ by

\[
\begin{array}{ll}
\varphi_1(E_{\{k,l\}
})=g_e\cdot e_{\{k,l\} }, & \text{ if } E_{\{k,l\}} \in \mathcal{S}_1,\\
\varphi_1(E_{\{k,l\} })=(-1)^{p_A(k)+p_A(u)+1}f_{u,l}e_{\{v,k\}
}+(-1)^{p_A(l)+p_A(u)+1}f_{u,k}e_{\{v,l\} },  & \text{ if } E_{\{k,l\}} \in \mathcal{S}_2,\\
\varphi_1(E_{\{k,l\} })=(-1)^{p_A(k)+p_A(v)+1}f_{v,l}e_{\{u,k\}
}+(-1)^{p_A(l)+p_A(v)+1}f_{v,k}e_{\{u,l\} },  & \text{ if } E_{\{k,l\}} \in \mathcal{S}_3.
\end{array}
\]
We need to prove that 
$d_1^{\mathbf{F}}(\varphi_1(\vec{v}))=g_e \cdot d_1^{\mathbf{G}}(\vec{v})$ for
any $\vec{v}\in G_1$. For a claw $A=\{v,u,k,l\}$ with center at $v$,
we have the relation \[(-1)^{p_A(k)+p_A(u)+1}f_{u,l}g_{v,k}
+(-1)^{p_A(l)+p_A(u)+1}f_{u,k}g_{v,l} = f_{k,l}g_{e}.\]
Similarily, for a claw $A=\{u,v,k,l\}$ with center at $u$,
we have the relation \[(-1)^{p_A(k)+p_A(v)+1}f_{v,l}g_{u,k}
+(-1)^{p_A(l)+p_A(v)+1}f_{v,k}g_{u,l} = f_{k,l}g_{e}.\]
This yields that 
$d_1^{\mathbf{F}}(\varphi_1(E_{\{i,j\} }))=g_{e} \cdot
d_1^{\mathbf{G}}(E_{\{i,j\} })$ for $E_{\{i,j\} } \in \mathcal{S}$.
So the mapping cone construction gives us a $S$-free presentation of
$L_G$ as
$$ F_2\oplus G_1\longrightarrow F_1\oplus G_0\longrightarrow F_0\longrightarrow L_G\longrightarrow 0.$$
Since $F_2\oplus
G_1 \simeq S^{\beta_{2}(S/L_G)}$ and $F_1\oplus G_0 \simeq S^n$, this is a minimal
free presentation. Hence the first syzygy of $L_G$ is minimally
generated by the images of basis elements under the map $\Phi: F_2
\oplus G_1 \longrightarrow F_1 \oplus G_0$,
where $\Phi = \begin{bmatrix} d_2^{\bf F}
& \varphi_1 \\
0 & -d_1^{\bf G} \end{bmatrix}$. Hence the assertion follows.
\end{proof}
 As a consequence of Theorem \ref{syzygy-tree} and Theorem \ref{syzygy-odd-unicyclic}, one can compute the defining ideal of symmetric algebra of LSS ideal of $G$, when $G$ is either a tree or an odd unicyclic graph.
Similarly, one can compute the first syzygy of parity binomial edge ideals of trees and odd unicyclic graphs and hence, the defining ideal of symmetric algebra of parity binomial edge ideals of trees and odd unicyclic graphs.

We now study linear type LSS ideals.  An ideal $I \subset A$ is said to be of {\it linear type} if  $Sym(I) \cong \R(I)$. Now, we recall the definition of $d$-sequence. 

\begin{definition}
Let $A$ be a commutative ring. Set $d_0 = 0$. A sequence of
elements $d_1 ,\dots ,d_n$ is said to be a $d$-sequence if 
$(d_0,d_1 ,\dots ,d_i)  : d_{i+1}d_j = (d_0,d_1 ,\dots ,d_i) : d_j$ for all $0 \leq i\leq n-1$ and for all $j\geq i+1$.
\end{definition}
We refer the reader to
 \cite{Hu80}  for more properties of
$d$-sequences.

\begin{theorem}\label{d-seq-LSS}
Let $G$ be a graph on $[n]$. Assume that $\K$ is an infinite field and char$(\K) \neq 2$. If $L_G$ is an almost complete intersection ideal,
then $L_G$ is generated by a homogeneous $d$-sequence. In particular, $L_G$ is of
linear type.
\end{theorem}
\begin{proof}Assume that $L_G$ is an  almost complete intersection ideal.
It follows from \cite[Proposition 5.1]{GM99}  that there exists a homogeneous set of generators $\{F_1,\ldots,F_{\mu(L_G)}\}$  of $L_G$ such that $F_1,\ldots, F_{\mu(L_G)-1}$ is a regular sequence in $S$. Since $J=(F_1,\ldots,F_{\mu(L_G)-1})$ is unmixed ideal, by  \cite[Theorem 4.7]{HMV89}, $J:F_{\mu(L_G)} =J:F_{\mu(L_G)}^2$. Hence, $L_G$ is generated by a homogeneous $d$-sequence $F_1,\ldots,F_{\mu(L_G)}$. The second assertion follows from \cite[Theorem 3.1]{Hu80}.
\end{proof}
In Theorem \ref{d-seq-LSS}, we assume that $\K$ is an infinite field, which is not a necessary condition. For, if $\K$ is a finite field with char$(\K) \neq 2$ and  $G$ is an odd unicyclic graph such that $L_G$ is almost complete intersection of type $(1)$ or type $(2)$ in Theorem \ref{odd-unicyclic-aci}, then it follows from the proof of Theorem \ref{odd-unicyclic-aci} that the generators of $L_G$ form a homogeneous $d$-sequence. Also, if $G$ is a bicyclic cactus graph such that $L_G$ is almost complete intersection, then   $L_G$ is generated by a homogeneous $d$-sequence (see proof of Theorem \ref{bicycilc-aci}).

Now, we prove that if $G$ is a bipartite graph such that $J_G$ is of linear type, then $G$ is $K_{2,3}$-free graph.
\begin{proposition}\label{rees-linear}
\begin{enumerate}
	\item Let $G$ be a bipartite graph such that $J_G$ is of linear type, then $G$ is $K_{2,3}$-free graph.
	\item If $K_4$ is an induced subgraph of $G$, then $J_G$ is not of linear type.
\end{enumerate}
\end{proposition}
\begin{proof} Let $\delta : S[T_{\{i,j\}} :
	\{i,j\} \in E(G)] \longrightarrow \R(J_G)$ is the map given by $\delta(T_{\{i,j\}}) =
	f_{i,j}t$. Then $J=\ker(\delta)$ is the  defining ideal of $\R(J_G)$.
\par (1) If possible, let $K_{2,3}$ be an induced subgraph of $G$. Without loss of generality, we may assume that $V(K_{2,3})=\{1,2,3,4,5\}$ and $E(K_{2,3})=\{\{1,3\},\{1,4\},\{1,5\},\{2,3\},\{2,4\},\{2,5\}$. Then it is easy to verify that $$x_3f_{1,4}f_{2,5}-x_4f_{1,3}f_{2,5}+x_3f_{1,5}f_{2,4}-x_5f_{1,3}f_{2,4}-x_4f_{1,5}f_{2,3}+x_5f_{1,4}f_{2,3}=0.$$ Thus, $F=x_3T_{\{1,4\}}T_{\{2,5\}}-x_4T_{\{1,3\}}T_{\{2,5\}}+x_3T_{\{1,5\}}T_{\{2,4\}}-x_5T_{\{1,3\}}T_{\{2,4\}}-x_4T_{\{1,5\}}T_{\{2,3\}}+x_5T_{\{1,4\}}T_{\{2,3\}} \in J$. By \cite[Corollary 2.3]{KM12}, $\beta_{2,3}(S/J_G)=0$, i.e. there is no linear relation in the first syzygy of $J_G$. Therefore, $F$ does not belong to  the module defined by first syzygy of $J_G$, which is a contradiction. 
\par (2) Assume that $V(K_4)=\{1,2,3,4\}$. Then $f_{1,2}f_{3,4}-f_{1,3}f_{2,4}+f_{1,4}f_{2,3}=0$ and hence $F=T_{\{1,2\}}T_{\{3,4\}}-T_{\{1,3\}}T_{\{2,4\}}+T_{\{1,4\}}T_{\{2,3\}} \in J$. The assertion follows, since $F$ is a quadratic homogeneous element.
\end{proof}

We have enough experimental evidence to pose the following
conjecture:
\begin{con}
	If $G$ is an odd unicyclic graph, then $L_G$ is of linear type.
\end{con}
Let $A$ be a Noetherian local ring with unique maximal ideal $\mathfrak{m}$. The \textit{fiber cone } of  an ideal $I$ is the ring $\mathcal{F}_I(A)= \R(I)/\mathfrak{m}\R(I) \cong \oplus_{k \geq 0} I^k/\mathfrak{m}I^{k}$. The \textit{analytic spread} of $I$ is the Krull dimension of the  fibre cone of $I$ and it is denoted by $l(I)$. Now, we prove that the fiber cone of LSS ideals of trees and  odd unicyclic graphs is a polynomial ring, i.e. $\mu(L_G)=l(L_G)$, if $G$ is either a tree or an odd unicyclic graph. 
\begin{theorem}\label{fiber-cone}
	Let $G$ be either a tree or an odd unicyclic graph on $[n]$. Then $\mu(L_G)=l(L_G)$.
\end{theorem}
\begin{proof}
	First, assume that $G$ is a tree. Then $Q_{\emptyset}(G)$ is a minimal prime of $L_G$ and $\h(Q_{\emptyset}(G))=n-1=\mu(L_G)$. Now, the assertion follows from \cite[Remark 2]{HMV89}. Assume that $G$ is an odd unicyclic graph. If char$(\K)=2$, then $L_G= \mathcal{I}_G$. Since $G$ is a non-bipartite graph, $\mathfrak{p}^+(G)$ is a minimal prime of $L_G$. Thus, $\h(\mathfrak{p}^{+}(G))=n= \mu(L_G)$. If $\sqrt{-1} \in \K$ and char$(\K) \neq 2$, then by Remark \ref{rmk-parity}, $\Psi(\eta(\mathfrak{p}^+(G)))$ is a minimal prime of $L_G$  such that $\h(\Psi(\eta(\mathfrak{p}^+(G)))) =n= \mu(L_G)$. Suppose that $\sqrt{-1} \notin \K$, then $I_{K_n}$ is a minimal prime of $L_G$ with $\h(I_{K_n})=n=\mu(L_G)$.  Hence,  by \cite[Remark 2]{HMV89}, the assertion follows.
\end{proof}
 
\vskip 2mm
\noindent
\textbf{Acknowledgements:} The author is grateful to his advisor A. V. Jayanthan for
his constant support, valuable ideas and suggestions. The author sincerely  thanks Santiago Zarzuela and  Sarang Sane for useful discussions. The  author thanks the National Board
for Higher Mathematics, India, for the financial support. The author also wishes to express his sincere
gratitude to the anonymous referees.

\bibliographystyle{plain}  
\bibliography{Reference}

\end{document}